\documentclass[english, 11pt, letterpaper, reqno]{article}

\usepackage[utf8]{inputenc}

\usepackage{microtype}
\usepackage{mathtools}
\usepackage{tikz}
\usepackage{arydshln}

\usepackage{geometry}               
\usepackage{graphicx}
\usepackage{bbm}
\usepackage{amssymb, amsmath, amsthm}
\usepackage{pdfsync}
\usepackage{hyperref}
\usepackage{booktabs}
\usepackage{authblk}
\usepackage{verbatim}

\newtheorem{theorem}{Theorem}
\newtheorem{conjecture}[theorem]{Conjecture}
\newtheorem{lemma}[theorem]{Lemma}

\newtheorem{corollary}[theorem]{Corollary}

\newcommand{\R}{\mathbb{R}}
\newcommand\extrafootertext[1]{%
    \bgroup
    \renewcommand\thefootnote{\fnsymbol{footnote}}%
    \renewcommand\thempfootnote{\fnsymbol{mpfootnote}}%
    \footnotetext[0]{#1}%
    \egroup
}

\date{\today}

\begin{document}

\title{When does  \texorpdfstring{$e^{-\lvert \tau\rvert }$}{exp(-|t|)} maximize Fourier extension for a conic section?\extrafootertext{Keywords: Sharp restriction theory, Fourier extension, maximizer, sphere, hyperboloid, cone.}\extrafootertext{$^1$giuseppe.negro@tecnico.ulisboa.pt; $^2$diogo.oliveira.e.silva@tecnico.ulisboa.pt; $^3$thiele@math.uni-bonn.de}}

 \author[1]{Giuseppe Negro}
   \author[2]{Diogo Oliveira e Silva}
   \author[3]{Christoph Thiele} 

\affil[1,2 ]{Departamento de Matem\'atica, Instituto Superior T\'ecnico, Universidade de Lisboa, Portugal.}
 \affil[3]{Mathematisches Institut, Rheinische Friedrich Wilhelms Universit\"at, Bonn, Germany. }

\maketitle
  
  \abstract{In the past decade, much effort has gone into understanding maximizers for Fourier restriction and extension inequalities. Nearly all of the cases in which maximizers for inequalities involving the restriction or extension operator have been successfully  identified can be seen as partial answers to the question in the title. In this survey, we  focus on recent developments in sharp restriction theory relevant to this question.
  We present results in the algebraic case for spherical and hyperbolic extension inequalities. We also discuss the use of the Penrose transform leading to some negative answers in the case of the cone. }

\section{Introduction}

We are interested in a family of maximization problems in Fourier extension theory for quadratic surfaces.
In \cite{FOS17}, this family of problems was formulated in a uniform manner, that we now recall.

Consider the standard cone $\mathbb{K}^{d+1}$
in $\mathbb{R}^{d+2}$, that is, the set of all $(\tau,\eta)$ with $\tau\in \mathbb{R}$, $\eta\in \mathbb{R}^{d+1}$, and $\tau^2=|\eta|^2$.
A conic section $S$ is the intersection of this
cone with a hyperplane $W$. After a rotation in the 
$\eta$-variables, which does not affect our object of interest, we may assume that the hyperplane $W$ has the equation
\begin{equation}\label{e:plane}
\alpha\tau+\beta\rho+\gamma=0
\end{equation}
for some real numbers $\alpha,\beta,\gamma$, where $\eta=(\rho,\xi)$ with $\rho\in $
$\mathbb{R}$ and $\xi\in  \mathbb{R}^{d}$.

A natural measure $\sigma$ on the conic section $S$ is given by
\begin{equation}\label{eq:sigma_measure_intro}
    \int_S f \textup d\sigma: = \int_{\mathbb{R}^{d+2}}
    f(\tau,\rho,\xi)\delta(\alpha \tau+\beta \rho+
    \gamma)\delta(\tau^2-\rho^2-|\xi|^2) \, \textup d\tau \textup d\rho \textup d\xi,
\end{equation}
where, say, $f$ is a continuous  function on the section $S$.
We denote by $L^2(\sigma)$ the Hilbert space closure of the space of functions satisfying 
\[\|f\|_{L^2(\sigma)}:= \left(\int_S |f|^2 \textup d\sigma\right)^{1/2} <\infty.\]
The Fourier extension of the function $f$ is the function
$$\widehat{f \sigma}(x):=  \int_{\mathbb{R}^{d+2}}
f(\tau,\rho,\xi)e^{- i \eta \cdot x} \delta(\alpha \tau+\beta \rho+
\gamma)\delta(\tau^2-\rho^2-|\xi|^2) \, \textup d\tau\textup d\rho\textup d\xi,$$
which we consider for $x\in W$.
More precisely, we are interested in the  quantity 
\begin{equation}\label{e:extfunctional}
\sup_f \|\widehat{f \sigma} \|_{L^p(W)},
\end{equation}
where the supremum is taken over all  $f$ in the closed unit ball $B$ of  $L^2(\sigma)$, and
\begin{equation}\label{e:pnormextension}
\|\widehat{f \sigma} \|_{L^p(W)}:= \left(\,\,
\int_{\mathbb{R}^{d+2}}
|\widehat{f\sigma}(x)|^p \delta(\alpha t+\beta r +
\gamma) \, \textup d t \textup d r \textup d x\right)^{1/p}
\end{equation}
In particular, we ask to identify maximizers $f$ in the ball $B$ for the quantity  
\eqref{e:pnormextension}, if they exist.
Often one formulates this problem in an equivalent way, as extremizing the quotient
\begin{equation}\label{e:quotient}
\sup_{\|f\|_{L^2(\sigma)}\neq 0}\|\widehat{f \sigma} \|_{L^p(W)}^p \|f \|_{L^2(\sigma)}^{-p}.
\end{equation}

It was observed in \cite{FOS17} that
 in many instances the function $f(\tau, \rho, \xi)=e^{-|\tau|}$ maximizes  \eqref{e:quotient}.
Indeed, this statement includes nearly all instances of known maximizers in sharp restriction theory, such as constant functions on 
spheres,  or maximizers which are gaussian with respect to projection measure in the
case of paraboloids, or exponential functions on the cone.
This survey summarizes some progress on the question ``For which parameters of the problem is 
 $e^{-|\tau|}$ a maximizer of  \eqref{e:quotient}?'',
with particular emphasis on progress since \cite{FOS17}.

Progress in identifying the maximizer
attaining the supremum \eqref{e:extfunctional}
was almost entirely made in the {\it algebraic case}, i.e., when
$p=2k$ is an even integer. 
We may then  write the $k$-th power of  \eqref{e:pnormextension} as
$$\|(\widehat{f\sigma})^k\|_2=(2\pi)^{\frac{d+1}2}\|({f\sigma})^{*k}\|_2,$$
where we used Plancherel's identity and wrote $(f\sigma)^{*k}$ for the successive convolution of $k$ copies of the measure $f\sigma$. The integral power inside the norm allows for some helpful exact manipulations of the expression.

Obviously, finiteness of the supremum  \eqref{e:extfunctional} is a necessary condition for it being attained by a maximizer. The range of exponents $p$ for which \eqref{e:extfunctional} is finite is 
known for all conic sections in all dimensions $d$; see Figure \ref{fig:PDP} for the algebraic case.

\begin{figure}[h]
    \centering
    \includegraphics[width=.8\textwidth]{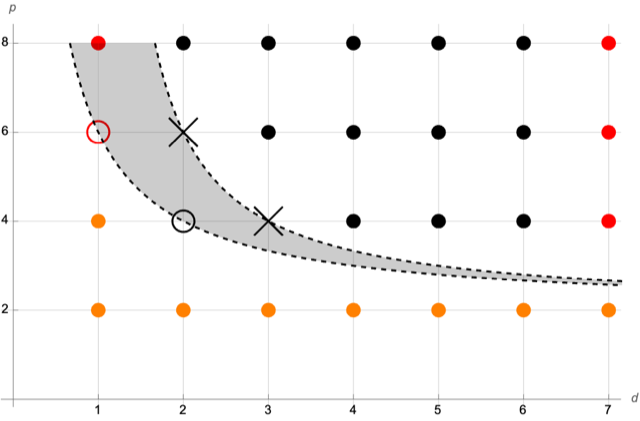}
    \caption{\footnotesize On the horizontal axis, dimension $d$. On the vertical axis, Lebesgue exponent $p$ of the $L^2(\sigma)\to L^p$ extension inequality. Circles ($\circ$) at $d\in\{1,2\}$ correspond to endpoint Stein--Tomas for spheres and Stein--Tomas for paraboloids, while crosses ($\times$) at $d\in\{2, 3\}$ correspond to  Stein--Tomas for cones. The shaded region between the two dashed curves correspond to hyperboloids. For spheres, black entries correspond to inequalities which are maximized by constants, while red entries correspond to inequalities which are conjecturally maximized by constants; orange entries correspond to situations in which Stein--Tomas does not hold, but other replacement inequalities such as Agmon--Hörmander may be available.}
    \label{fig:PDP}
\end{figure}

Letting $(\alpha,\beta,\gamma)=(1,0,-1)$,  equation~\eqref{e:plane} reduces to $\tau=1$, so the section $S$ corresponds to the unit sphere $\mathbb S^d$. In this case, the function $e^{-|\tau|}$ is constant. 
The measure on $\mathbb S^d$ defined by~\eqref{eq:sigma_measure_intro} is then nothing but the usual surface measure.
The supremum in~\eqref{e:extfunctional} is finite for $p\geq 2+\frac4d$, by the Stein--Tomas inequality \cite{St93, To75} on $\mathbb S^{d}$. We call $p=2+\frac4d$ the endpoint Stein--Tomas exponent.
The first result identifying constants as maximizers for the sphere was proved by Foschi \cite{Fo15} in the algebraic case $p=4$ when $d=2$. This corresponds to the sharp endpoint Stein--Tomas restriction theorem on $\mathbb S^2\subset\mathbb R^3$.
Remarkably, the case of the endpoint Stein--Tomas exponent for $\mathbb S^1\subset\R^2$,
the only other case where this endpoint is algebraic for the sphere, remains open. Much effort has gone into studying this case. In particular, a modification of Foschi's approach has been suggested
but only led to partial progress. In Section \ref{s:circle}, we discuss some rigorous numerical computations to verify that constants maximize \eqref{e:extfunctional} for $\mathbb S^1$, among 
those functions in the unit ball of $L^2(\sigma)$ which have Fourier modes up to degree $120$. These simulations also suggest an asymptotic behaviour for
large Fourier modes, which in the future may help complement the numerical 
work with an analytic proof.

In Section \ref{s:spheres}, we address the case of higher dimensional spheres $\mathbb S^d$, $d\geq 2$. We describe in some detail a common approach to all the black entries on Figure \ref{fig:PDP}, corresponding to $2\leq d\leq 6$, present a particular instance of a sharp extension inequality on $\mathbb S^7$, and conclude with an open problem.

On the unit sphere $\mathbb S^{d}$, the Stein--Tomas condition excludes $p=2$
in all dimensions. Following Agmon and  H\"ormander \cite{AH76}, we therefore also consider the modified version of~\eqref{e:extfunctional}
\begin{equation}\label{e:extfunctional_mod}
\sup_f R^{-1}\|\widehat{f \sigma} 1_{B_R}\ \|_{L^2(W)},
\end{equation}
where $B_R$ denotes the ball in $W\equiv\mathbb R^{d+1}$ of radius $R$ about the origin. Even though this
functional is related to an endpoint result for the classical  proof of the Stein--Tomas inequality via complex interpolation, for some values of the parameters $d, R$ the supremum in~\eqref{e:extfunctional_mod} is maximized when $f$ is constant, while for other values of the parameters it is maximized by other, less symmetric functions. This is an instance of \emph{symmetry breaking}. Section \ref{s:agmonh} surveys
this in more detail.

In Section \ref{s:cones}, we specialize the hyperplane~\eqref{e:plane} to the limiting case $\rho=0$, i.e. $(\alpha,\beta,\gamma)=(0,1,0)$. The resulting conic section $\tau^2=\lvert \xi\rvert^2$, where $\xi\in\mathbb R^d$, is identified with the cone $\mathbb{K}^{d}$. The measure $\sigma$ of~\eqref{eq:sigma_measure_intro} reduces to the Lorentz-invariant measure on $\mathbb{K}^{d}$, which is unique up to scalar multiplication. The quantity~\eqref{e:extfunctional} is finite if and only if $p=2\frac{d+1}{d-1}$, by the Strichartz estimate for the wave equation \cite{St77}. 
Foschi~\cite{Fo07} showed that $f(\tau, \xi)=e^{-\lvert\tau\rvert}$ maximizes this estimate when $(d,p)=(3,4)$, and he formulated the natural conjecture that the same should be true in arbitrary dimensions. This turned out to be false for even $d$; we will discuss the proof, which is based on the Penrose conformal compactification of Minkowski spacetime. 

Some related extension problems are also meaningful in this cone context. The aforementioned~\cite{Fo07}  establishes that, on the half-cone $\mathbb{K}^{d}_+=\mathbb{K}^{d}\cap \{ \tau>0 \}$, equipped with the natural restriction $\sigma_+$ of the Lorentz-invariant measure $\sigma$, the function $f(\tau, \xi)=e^{-\lvert\tau\rvert}$ is a maximizer for $d\in\{2,3\}$. We will briefly discuss the conjecture that this should hold in arbitrary dimension, which is still open. Finally, we will consider the case of weighted versions of the measure $\sigma_+$, corresponding to Strichartz estimates with a different Sobolev regularity of the initial data.

Hyperboloids, corresponding to $(\alpha,\beta,\gamma)=(0,1,1)$ in \eqref{e:plane}, capture features from both spheres  (Sections \ref{s:circle}--\ref{s:agmonh}) and cones  (Section \ref{s:cones}), but new phenomena arise -- most notably, maximizers sometimes fail to exist.
In spite of this, optimal constants have been determined for all the algebraic endpoint inequalities, and concentration-compactness tools have been used together with bilinear restriction theory in order to shed some light on the behaviour of arbitrary maximizing sequences in the case of non-endpoint inequalities. In Section \ref{s:hyperbolic}, we survey both of these types of results.

While in Fourier extension theory one also studies finiteness of \eqref{e:extfunctional}
when $f$ ranges over the unit ball of $L^q(\sigma)$ for $q\neq 2$, there is little progress in identifying $e^{-|\tau|}$ as the maximizer for $q\neq 2$ and we therefore focus in this survey on $q=2$.

\subsection{Remarks and further references}
Following \cite{FOS17} and \cite[Introduction]{OS17}, the style of this survey is admittedly informal. In particular, some objects will not be rigorously defined, and several results will not be precisely formulated. Most of the material is not new, exceptions being some parts of Section \ref{s:circle}, the main result in \S \ref{s:1cone}, and a few observations which we could not find in the literature. The subject is becoming more popular, as shown by the increasing number of works that appeared in the last decade. We have attempted to provide a rather complete set of references, which includes several interesting works on spheres \cite{BV20, BBI15, COS15, COSS19b, CS12a, CG22, FS22, FLS16, MOS21, Sh16}, cones \cite{BJ15, B10, FK00, NOSST22, OR13, Q13, Ra12}, hyperboloids \cite{COS22, DMPS18, J14, OR13, OR14, Qu22}, curves \cite{BS20, FS18, HS12, OS14, OS18, Sh09a}, paraboloids \cite{BV07, BBCH09, BBFGI18, BBJP17, Ca09, CQ14, DMR11, Go17, Go19, GZ20, HZ06, K03, OT98, Sh09, St20, WZ21} and perturbations \cite{DY22, JPS10, JSS14, OSQ18, OSQ20, Ta21a,Ta21b} that will not be discussed here. Given its young age, there are plenty of open problems in the area. We provide some more.

\section{The circle} 
\label{s:circle}
The endpoint Stein--Tomas exponent for the sphere $\mathbb S^d$ is an even integer in dimensions $d\in\{1,2\}$.
If $d=2$, then the endpoint exponent is $p=4$, and 
this is the only case in which constant functions are known to maximize the  Fourier extension map for the endpoint Stein--Tomas exponent, thanks to the argument of Foschi \cite{Fo15}. Remarkably, 
the question remains open in dimension $d=1$, where the exponent is $p=6$. A  modification of Foschi's argument for $d=1$ was proposed in \cite{CFOST17}, and this approach was substantiated by extensive numerical evidence in \cite{OSTZK18} and \cite{BTZK20}. In this section, we discuss this approach to the circle problem.

We first recall Foschi's argument \cite{Fo15} for $d=2$ in a form reflecting more recent insights. Let the measure $\sigma$ be defined by $\sigma(x)=\delta(|x|^2-1)$. 
We aim to show
\begin{equation}\label{e:foschi2}
\|\widehat{f\sigma}\|_4^4\le \|\widehat{\sigma}\|_4^4  
\end{equation}
for all $f$ that have the same $L^2(\sigma)$-norm as the constant function $1$. 

Basic steps that 
are recalled in Section \ref{s:spheres} allow to restrict attention to $f$ that are real and antipodally symmetric, so that also $u:=\widehat{f\sigma}$ is real and even.
As $\widehat{u}$ is supported on the unit sphere, $u$ satisfies the  Helmholtz equation, 
$$u+\Delta u=0.$$ 
This equation is used in the first key step that is sometimes referred to as the {\it magic identity}. It expresses the left-hand side of  \eqref{e:foschi2} in a  way that later removes a singularity of $\sigma*\sigma$ at the origin. 
The derivation of the magic identity uses partial integration, which is
justified because $u$ and all its derivatives are in $L^4$; see \cite[Prop.\@ 6]{CNOS21}. Replacing one factor of $u$
with Helmholtz, we obtain
 for the left-hand side of  \eqref{e:foschi2}
\begin{equation}\label{eq:magic_identity}
-\int_{\mathbb{R}^3} (\Delta u)u^3 =\int_{\mathbb{R}^3} \nabla u \cdot \nabla (u^3)
=3\int_{\mathbb{R}^3} |\nabla u|^2u^2=\frac 34 \int_{\mathbb{R}^3} |\nabla u^2|^2=-\frac 34 \int (\Delta u^2)u^2.
\end{equation}
Expressing this again in terms of $f$, and proceeding analogously for the constant function, we reduce 
\eqref{e:foschi2} 
to
$$\Lambda(f,f,f,f)\le \Lambda({\bf 1},{\bf 1},{\bf 1},{\bf 1})$$
with $\Lambda$ defined by
$$\Lambda(f_1,f_2,f_3,f_4)=\int_{(\mathbb{R}^3)^4} |x_1+x_2|^2 \delta\left(\sum_{k=1}^4  x_k\right) \prod_{j=1}^4 f_j\sigma (x_j)\textup d x_j.
$$
The second key step is a positivity argument reminiscent of the Cauchy--Schwarz inequality.
We use the inequality 
\begin{equation}\label{e:ab}
2ab\le a^2+b^2
\end{equation}
for real numbers  $a$ and $b$ chosen to be $f(x_1)f(x_2)$
and $f(x_3)f(x_4)$ at every point $x_1,x_2,x_3,x_4$ and obtain by positivity of the integral kernel
\begin{equation}\label{eq:CScreative}
2\Lambda(f,f,f,f)\le \Lambda(f^2,f^2,{\bf 1},{\bf 1})+ \Lambda({\bf 1},{\bf 1},f^2,f^2).
\end{equation}
Using  symmetry of $\Lambda$ thanks to $|x_1+x_2|=|x_3+x_4|$, we are reduced to showing
$$\Lambda(f^2,f^2,{\bf 1},{\bf 1})\le \Lambda({\bf 1},{\bf 1},{\bf 1},{\bf 1}).$$
We write $f^2={\bf 1}+g$, where $g\sigma$ has integral zero.
In the integral expression for  $\Lambda({\bf 1},{\bf 1},{\bf 1},g)$, the measure  $g\sigma$ is integrated against a radially symmetric function and thus this integral is zero.
Using expansion and symmetry of $\Lambda$ again, we are  reduced to showing the third key step, which is the inequality
\begin{equation}\label{e:3rdkey}
\Lambda(g,g,{\bf 1},{\bf 1})\le 0.
\end{equation}
This relies on the important calculation, see e.g.  \cite{OSQ21a} or Lemma \ref{lem:2-fold} below, that $(\sigma*\sigma)(x)$ is a positive scalar multiple of $$ |x|^{-1}1_{|x|\le 2}. $$
This convolution appears in the expression for $\Lambda(g,g,{\bf 1},{\bf 1})$ by integrating out the last two variables. Thanks to the support of the remaining functions, we may omit the indicator $1_{|x|\le 2}$
and reduce $\eqref{e:3rdkey}$ to
\begin{equation}\label{e:ggoneone}
 \int_{(\mathbb{R}^3)^2}   \frac{|x_1+x_2|^2}{|x_1+x_2|}
g\sigma(x_1) g\sigma(x_2)\textup dx_1 \textup d x_2 \le 0.
\end{equation}
It is here, as a result of the magic identity \eqref{eq:magic_identity}, that the singularity 
at $|x_1+x_2|=0$ is cancelled.

Define the analytic family $h_s$ of 
tempered distributions 
on $\mathbb{R}^3$
 by
\begin{equation}\label{e:homdis}
h_s(\phi)=\Gamma\left(\frac{s+3}2\right)^{-1}\int_{\mathbb{R}^3}|x|^
{s} \phi(x) {\textup d x}.
\end{equation}
The integral \eqref{e:homdis} is well-defined  for 
$\Re(s)>-3$ and positive when
$\phi$ is a gaussian.
The inequality \eqref{e:ggoneone} is then equivalent to 
\begin{equation}\label{e:h1}
h_1(g\sigma* g \sigma)\le 0.
\end{equation}

The distribution $h_s$ is rotationally symmetric and homogeneous 
of degree $s$, and it is up to positive scalar uniquely determined by 
these symmetries
and positivity on the gaussian.
As the Fourier transform  $\widehat{h}_s$ also has rotational symmetry and dilation symmetry
with  degree of homogeneity $-3-s$, we have for $-3<\Re(s)<0$
\begin{equation}\label{e:fcte}
a\widehat{h}_s= h_{-3-s}
\end{equation}
for some positive constant $a$.
Analytic continuation with \eqref{e:homdis} and \eqref{e:fcte} allows to define $h_s$ for all complex numbers $s$. By unique continuation,
  $h_s(\phi)$ is expressed by  \eqref{e:homdis} whenever $\phi$ vanishes of sufficiently high order at $0$ so that the integral is absolutely convergent.

By Plancherel,  we reduce \eqref{e:h1} to
\begin{equation}\label{e:hminus4}
h_{-4}( (\widehat{g\sigma})^2)\le 0 .\end{equation}
As $(\widehat{g\sigma})^2$  vanishes to second order at the origin, the pairing with $h_{-4}$ is given by the expression \eqref{e:homdis}.
Inequality \eqref{e:hminus4} follows, because $(\widehat{g\sigma})^2$ is nonnegative and  $\Gamma(-1/2)<0$.
This concludes our discussion of the case $d=2$.

We turn to the case of the circle, $d=1$. We again assume $f$ smooth, real-valued, and antipodally symmetric. The endpoint Stein--Tomas exponent is $p=6$ and one may deduce a magic identity 
$$\|u\|_6^6=-\frac 59 \int_{\mathbb{R}^3} (\Delta u^3) u^3.$$
However, we aim to cancel a singularity of $\sigma*\sigma*\sigma$ at $|x|=1$,
so we consider the identity in the form
$$\frac 45 \|u\|_6^6=- \int_{\mathbb{R}^3} (u^3+\Delta  u^3)u^3.$$
The second step in Foschi's program  suggests the inequality
\begin{equation}\label{e:conjcs}\Lambda(f,f,f,f,f,f)\le \Lambda(f^2,f^2,f^2,{\bf 1},{\bf 1},{\bf 1})
\end{equation}
with
$$\Lambda(f_1,f_2,f_3,f_4,f_5,f_6):=\int_{(\mathbb{R}^2)^6} (|x_1+x_2+x_3|^2-1) \delta\left(\sum_{k=1}^6  x_k\right) \prod_{j=1}^6 f_j\sigma (x_j)\textup d x_j.
$$
Inequality \eqref{e:conjcs} does not follow as above from a pointwise application of 
the elementary pointwise inequality \eqref{e:ab} because the integral kernel is not positive.
Indeed, inequality \eqref{e:conjcs} is not known to be true in the required level of generality.
Assuming \eqref{e:conjcs} nevertheless for the moment, it remains to show the third key step,
\begin{equation}\label{e:tril}\notag \Lambda(f^2,f^2,f^2,{\bf 1},{\bf 1},{\bf 1})\le
\Lambda({\bf 1},{\bf 1},{\bf 1},{\bf 1},{\bf 1},{\bf 1}).
\end{equation}
This inequality is the main result of \cite{CFOST17}. The 
technique used in
\cite{CFOST17} is to expand the function $f^2$ on the circle into Fourier series and use careful computations from \cite{OST17} for the tensor coefficients of $\Lambda$ in the Fourier basis. These arguments are very technical and go beyond the scope of this survey.

The question whether constant functions maximize the Stein--Tomas inequality on the circle is thus reduced to proving
\eqref{e:conjcs}. 
Additional motivation for this conjecture is that thanks to the antipodal symmetry of $f$ the values of the integrand on the negative and positive parts are correlated. We formulate an even more general conjecture.
\begin{conjecture}
\label{c:csconjecture}
The quadratic form  
$$Q(f)=\Lambda(f^2,{\bf 1})-\Lambda(f,f),$$ 
with
$\Lambda(f_1,f_2)$ defined as
 \[\int_{(\mathbb{R}^2)^6} (1-|x_1+x_2+x_3|^2) \delta\left(\sum_{k=1}^6  x_k\right)f_1(x_1,x_2,x_3)f_2(x_4,x_5,x_6) \prod_{j=1}^6 \sigma (x_j)\textup dx_j
\]
is positive semi-definite
on the space of real-valued  functions $f\in L^2((\mathbb{S}^1)^3)$ that are antipodally symmetric in each variable.

\end{conjecture}

Positive definiteness of $Q$ has been verified on a large finite dimensional subspace of $f\in L^2((\mathbb{S}^1)^3)$ in \cite{BTZK20},
using a large computing cluster
and
extending previous numerical results in \cite{OSTZK18}. 
More specifically, identifying $\R^2$ with the complex plane $\mathbb C$, we  expand $f$ in $L^2((\mathbb{S}^1)^3)$ into Fourier series
\begin{equation}\label{e:fseries}\notag
f(x_1,x_2,x_3)=\sum_{n\in \mathbb{Z}^3} \widehat{f}(n_1,n_2,n_3) x_1^{n_1} x_2^{n_2} x_3^{n_3}.
\end{equation}
 Note that antipodal symmetry translates into vanishing of Fourier coefficients with at least one odd index.
The following theorem 
is proven in
 \cite{BTZK20}.
\begin{theorem}[\cite{BTZK20}]\label{t:120}
Let $f\in L^2((\mathbb{S}^1)^3)$ 
so that the Fourier coefficient $\widehat{f}(n_1,n_2,n_3)$ is zero if
one of the numbers $n_1$, $n_2$, $n_3$ is odd or larger than $120$. Then
$Q(f)\ge 0$.
\end{theorem}

Foschi's program then shows that constants maximize the endpoint Stein--Tomas inequality on the circle
among all functions with Fourier modes up to degree $120$.

The numerical computations also suggest an asymptotic behaviour of $Q$ for large frequencies.
This could be helpful for a resolution of Conjecture \ref{c:csconjecture}. We  summarize some of these observations in the remainder of this section, which is best read with the displays from \cite{BTZK20} at hand.

The quadratic form $Q$ is symmetric under joint rotation of the variables of $f$, that is, 
$Q(f)=Q(g)$ whenever
\[g(x_1,x_2,x_3)=f(\omega x_1,\omega x_2,\omega x_3)\]
for a complex number $\omega$ of modulus one.  This rotation action decomposes $L^2((\mathbb S^1)^3)$ into mutually orthogonal eigenspaces, and it 
suffices to prove positive semi-definiteness of $Q$ on
each eigenspace separately.
Numerical simulations confirm that understanding of the matter hinges on understanding of the main eigenspace $V$ spanned by  Fourier modes which satisfy
\begin{equation}\label{e:mainblock}n_1+n_2+n_3=0.
\end{equation}
The index set of the Fourier modes with even $-120\le n_1,n_2,n_3\le 120$ satisfying \eqref{e:mainblock}
is schematically depicted via black dots in the hexagon on the left in Figure \ref{fig:MRE}.
The colored structures in the plot show the locus of relatively large coefficients of
a typical row of the symmetric quadratic form $Q$. There is a cluster of very large values 
near the diagonal element, multiplied  by the six fold permutation symmetry of the indices.
Further relatively large coefficients appear near a circular structure. All remaining coefficients
are quite small. The occurrence of this circular structure is observed but as of yet poorly understood.

\begin{figure}
    \centering
        \includegraphics[width=.49\textwidth]{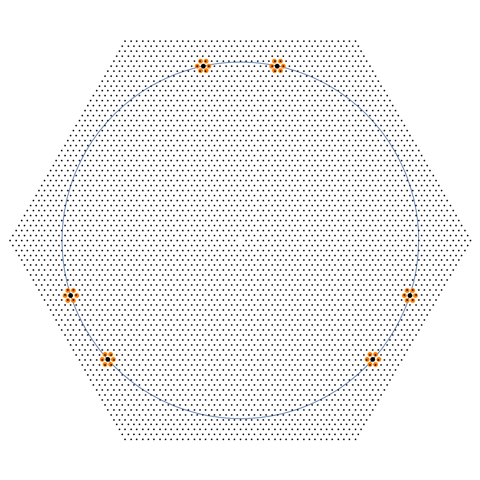}
    \includegraphics[width=.49\textwidth]{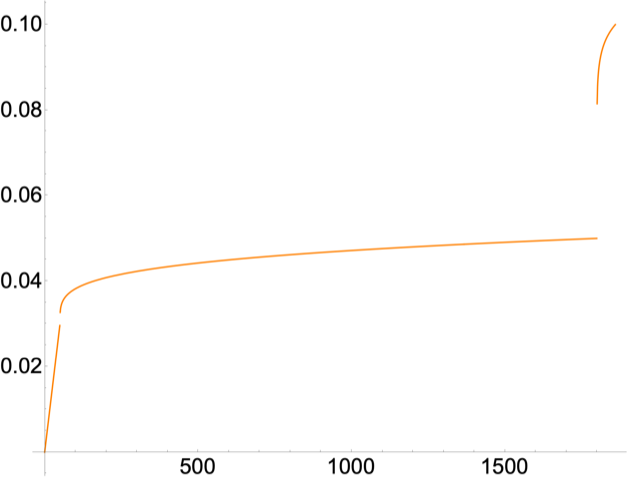}
    \caption{\footnotesize Fourier modes and eigenvalues; see \cite{OSTZK18} and especially \cite[Figure 6]{BTZK20}.}
    \label{fig:MRE}
\end{figure}

On the right in Figure \ref{fig:MRE}, the eigenvalues of $Q$ on $V$
 are schematically depicted in increasing order. They
fall into 
 three regions: small, intermediate, and  large eigenvalues.
The small eigenvalues  are 
the main difficulty in proving Conjecture \ref{c:csconjecture}. Numerical simulations suggest that these eigenvalues are caused by 
highly localized functions, 
$$f_\epsilon(x_1,x_2,x_3)=
\prod_{j=1}^3 \exp(-\epsilon^{-2}\Re(x_j)^2),$$
for small $\epsilon$, which approximate a sum of antipodal Dirac deltas.
Denote the orthogonal projection of $f_\epsilon$ to the eigenspace $V$ by $g_\epsilon$. Jiaxi Cheng, a graduate student at Bonn, theoretically observed  
the asymptotic behaviour
\begin{equation}\label{e:asympt}
\left|Q(g_\epsilon)\|g_\epsilon\|_{L^2(\sigma)}^{-2}-c \log(\epsilon)\epsilon^2\right|=o(\log(\epsilon)\epsilon^2)
\end{equation}
using Taylor expansion near the essential
support of $g_\epsilon$.
This models most if not all of the small eigenvalues that were observed in the numerical
simulations in \cite{BTZK20}.
An asymptotic behaviour of the smallest eigenvalue was explicitly suggested 
in \cite{BTZK20} and interpreted as  proportional to $ \epsilon^{1.74}$,
which in the range of observations
$1/20<\epsilon<1/120$ is nearly proportional
to $\log(\epsilon)\epsilon^2$. The theoretically discovered constant 
$c$ in \eqref{e:asympt} by Cheng is consistent with the numerical observations.

The functions $g_\epsilon$  inherit an approximate radial symmetry from the radial symmetry of a Dirac delta and the approximating gaussian. This results in an approximate
radial symmetry of the Fourier coefficients of the eigenfunctions with small eigenvalues, parameterized by the points in the hexagon in Figure \ref{fig:MRE}. This was
observed  in \cite{BTZK20}.
Further computations by Jiaxi Cheng show that these eigenfunctions 
are properly modeled by eigenfunctions of a Schr\"odinger operator in the plane with infinite hexagonal well or, approximately, circular well. The latter eigenfunctions are explicitly known, they are radial functions with radial profile given by a scaled Bessel function $J_0$ so that the spherical well falls on a zero of the scaled Bessel function. This is matching the numerical findings in 
\cite[Fig.\@ 8]{BTZK20}, where the eigenfunctions with the five smallest eigenvalues are displayed.

The above discussion gives hope that one can find a good analytic understanding of the small eigenvalues, say on a  suitably defined space of radial functions.
On the orthogonal complement of such space, one mainly faces the 
intermediate eigenvalues displayed in 
Figure \ref{fig:MRE}, which appear to come from a small perturbation of a multiple of the  identity matrix. This is confirmed by the typical matrix row of $Q$, schematically depicted in 
Figure \ref{fig:MRE}, where the blue circle and the orange dots indicate the loci of  large matrix entries, and are more precisely displayed in 
\cite[Figs.\@ 3--4]{BTZK20}. The circular structure disappears when projected
to the orthogonal complement of the radial functions, while the orange clusters, up to the symmetry group of three elements, represent a nearly diagonal perturbation of the identity matrix.

The small group of large eigenvalues in 
Figure \ref{fig:MRE} is not well understood, but might be an artifact of the precise truncation to the hexagon in the numerical calculations and might be suppressed in a more spherically symmetric or smoothened approach.

\section{Spheres} 
\label{s:spheres}
In this section, we consider the case of higher dimensional spheres $\mathbb S^{d}\subset\R^{d+1}$, $d\geq 2$, equipped with the usual surface measure $\sigma$ as in \eqref{eq:sigma_measure_intro} or, equivalently, as in the paragraph before \eqref{e:foschi2}.
The five black $L^2\to L^4$ entries on Figure \ref{fig:PDP}, corresponding to cases for which constants are global maximizers, have been previously surveyed in \cite{FOS17}. Above them  lie infinitely many $L^2\to L^{2k}$ estimates which have been recently put in sharp form. As in \eqref{e:quotient}, we define the functional
\[\Phi_{d,p}[f]:=\|\widehat {f\sigma}\|^p_p\|f\|_{L^2(\mathbb S^d)}^{-p},\]
and present the main result in \cite{OSQ21a,OSQ21b}.

\begin{theorem}[\cite{OSQ21a,OSQ21b}]\label{thm:spheres}
Let $d\in\{2,3,4,5,6\}$ and $p\geq 6$ be an even integer. Then constant functions are the unique real-valued maximizers of the functional $\Phi_{d,p}$.
The same conclusion holds for $d=1$ and even $p>6$ if constants maximize $\Phi_{1,6}$.
\end{theorem}

The purpose of this section is four-fold. 
Firstly, we briefly discuss the proof of Theorem \ref{thm:spheres} in the particular but representative case when $(d,p)=(2,6)$.
Secondly, we describe the extra ingredients which are needed in order to obtain sharp $L^2\to L^{2k}$ estimates for higher $k\geq 4$.
Thirdly, we characterize the class of complex-valued maximizers of $\Phi_{d,p}$, for $d,p$ in the range covered by Theorem \ref{thm:spheres}. 
Finally, we suggest a way to go beyond this range, by presenting a sharp extension inequality on $\mathbb S^7$ in the weighted setting.

\subsection{The case $(d,p)=(2,6)$}\label{sec:specialcase}
We abbreviate notation by writing  $\Phi_p:=\Phi_{2,p}$ and 
\[{\bf T}_p:=\sup \{\Phi_p[f]^{1/p}: 0\neq f\in L^2(\mathbb S^2)\}.\]
The proof naturally splits into five steps, which use tools from the calculus of variations, symmetrization, operator theory, Lie theory, and probability. 
We present them next, and then see how they all come together.
\subsubsection{Calculus of variations}
The existence of maximizers for $\Phi_{6}$ is ensured by \cite{FVV11}.  
Let $f$ be one such maximizer, normalized  so that $\|f\|_{L^2}=1$. 
   Consider the extension operator $\mathcal{E} (f):=\widehat{f\sigma}$ with adjoint given by $\mathcal{E}^*(g):={g}^\vee|_{\mathbb S^{2}}$.
Then the operator norm can be estimated as follows:
\begin{align*}
\|\mathcal{E}\|_{L^2\to L^6}^6
&{=\|\mathcal{E} (f)\|_{L^6(\R^3)}^6=\langle |\mathcal{E} (f)|^{4}\mathcal{E} (f), \mathcal{E} (f) \rangle}
{=\langle\mathcal{E}^*(|\mathcal{E} (f)|^{4}\mathcal{E} (f)),f\rangle_{L^2(\mathbb S^{2})}}\\
&{\leq \|\mathcal{E}^*(|\mathcal{E} (f)|^{4}\mathcal{E} (f))\|_{L^2(\mathbb S^{2})}}
{\leq \| \mathcal{E}^*\|_{L^{6/5}\to L^2} \||\mathcal{E} (f)|^{4}\mathcal{E} (f) \|_{L^{6/5}(\R^3)}}\\
&{= \| \mathcal{E}^*\|_{L^{6/5}\to L^2} \|\mathcal{E} (f)\|_{L^6(\R^3)}^{5}}
{=\|\mathcal{E}\|_{L^2\to L^6}^6}.
\end{align*}
Thus all inequalities are equalities, and in particular equality in the Cauchy--Schwarz step above yields the {\it Euler--Lagrange equation},
\[(|\widehat{f\sigma}|^4\widehat{f\sigma})^\vee|_{\mathbb S^{2}}=\lambda f,\]
which, in convolution form, reads as follows:
\begin{equation}\label{eq:EL}
(f\sigma\ast f_\star\sigma\ast f\sigma\ast f_\star\sigma\ast f\sigma)|_{\mathbb S^{2}} = (2\pi)^{-3} \lambda f;
\end{equation}
{here $f_\star:=\overline{f(-\cdot)}$.} A bootstrapping procedure can then be used to show that
$f$, and indeed any $L^2$ solution of \eqref{eq:EL}, is $C^\infty$-smooth. We omit the details and refer the interested reader to \cite{OSQ21b}, which extends the main result of \cite{CS12b} to the higher dimensional setting of even exponents.

\subsubsection{Symmetrization}
Since $p=6$ is an even integer, the problem is inherently positive, in the sense that nonnegative maximizers exist. In fact,
\[\|f\sigma\ast f\sigma\ast f\sigma\|_{L^2(\R^3)}\leq\||f|\sigma\ast|f|\sigma\ast|f|\sigma\|_{L^2(\R^3)}.\]
Defining
 $f_\sharp:=\sqrt{\frac{|f|^2+|f_\star|^2}2}$,
 we also have the following monotonicity under antipodal symmetrization\footnote{An analogous inequality to \eqref{eq:antipodal_symmetrization_sphere} holds in the non-algebraic case; see \cite[Prop.\@ 6.7]{BOSQ20}, and also~\cite[Prop.~3.1]{GN20} for the case of the cone and the underlying wave equation.} which can be readily verified via a creative application of the Cauchy--Schwarz inequality in the spirit of \eqref{eq:CScreative}:
\begin{equation}\label{eq:antipodal_symmetrization_sphere}
    \|f\sigma\ast f\sigma\ast f\sigma\|_{L^2(\R^3)}\leq\|f_\sharp\sigma\ast f_\sharp\sigma\ast f_\sharp\sigma\|_{L^2(\R^3)}.
\end{equation}
We conclude that
\[{\bf T}_6=\max \{\Phi_6[f]^{1/6}: 0\neq f\in C^\infty(\mathbb S^2),  f \text{ is nonnegative and even} \},\]
an important simplification which will be crucial in the sequel.

\subsubsection{Operator Theory}
We now explore some of the compactness inherent to the problem.
Associated to a given $f\in L^2(\mathbb S^2)$, consider the integral operator $T_f: L^2(\mathbb{S}^2) \to L^2(\mathbb{S}^2)$ defined by
 \begin{equation}\label{eq:TfDef}
 T_f(g)(\omega):=(g\ast K_f)(\omega)=\int_{\mathbb S^2} g(\nu) K_f(\omega-\nu)\textup d\sigma(\nu),
 \end{equation}
which acts on functions $g\in L^2(\mathbb S^2)$ by  convolution with the kernel 
\[K_f(\xi):=(|\widehat{f\sigma}|^4)^\vee(\xi)=(2\pi)^3(f\sigma\ast f_\star\sigma\ast f\sigma\ast f_\star\sigma)(\xi).\]
Note that the Euler--Lagrange equation \eqref{eq:EL} can be written as the (nonlinear) eigenfunction equation $T_f(f)=\lambda f$.
The kernel $K_f$ defines a bounded, continuous function on $\R^3$ which satisfies 
 $K_f(\xi)=\overline{K_f(-\xi)}$, for all $\xi$,
and crucially $K_f(0)=\|\widehat{f\sigma}\|_4^4$.
Correspondingly, the operator $T_f$ is self-adjoint and positive definite. 
In fact, one can check that $T_f$ is trace-class and that its trace is given by
\begin{equation}\label{eq:trace}
\textup{tr}(T_f)=4\pi\|\widehat{f\sigma}\|_4^4.
\end{equation}
This is a consequence of Mercer's theorem, the infinite-dimensional analogue of the well-known statement that any positive semi-definite matrix is the Gram matrix of some set of vectors.

\subsubsection{Lie Theory}
We proceed to discuss the symmetries of the problem.
The set of  $3\times 3$ orthogonal matrices with unit determinant forms the special orthogonal group SO$(3)$, with Lie algebra $\frak so(3)$.
As a preliminary observation, we note that
the exponential map $\exp:\frak{so}(3)\to\textup{SO}(3), A\mapsto e^A$, is surjective onto $\textup{SO}(3)$, and that the functional $\Phi_6$ is rotation- and modulation-invariant. In other words,
\[\Phi_{6}[f\circ e^{tA}]=\Phi_{6}[f]=\Phi_6[e^{i\xi\cdot} f],\]
for all $(t,A)\in\R\times\frak{so}(3)$ and $\xi\in\R^3$.
As we shall now see, these symmetries naturally give rise to new eigenfunctions for the operator $T_f$ defined in \eqref{eq:TfDef}.
Consider the vector field $\partial_A$ acting on sufficiently smooth functions $f:\mathbb S^2\to\mathbb C$ via 
\[\partial_A f:=\frac{\partial}{\partial t}\Big\vert_{t=0} (f\circ e^{tA}).\] 
We have the following key lemma, where we write $\omega=(\omega_1,\omega_2,\omega_3)\in \mathbb S^2$, and by $\omega_j f$ we mean the function defined via $(\omega_j f)(\omega)=\omega_j f(\omega)$.
   \begin{lemma}\label{lem:key}
Let $f:\mathbb S^2\to\R$ be non-constant, such that $f_\star=f\in C^1(\mathbb S^{2})$ and $\|f\|_{L^2}=1$.
 Assume $T_f(f)=\lambda  f$.
Then
\begin{align*}
T_f(\partial_A f)&={{\tfrac{\lambda}{5}}} \partial_Af,\,\,\,\, \text{ for every }A\in\frak{so}(3),\\
T_f(\omega_j f)&={{\tfrac{\lambda}{5}}} \omega_j f,\,\, \text{ for every }j\in\{1,2,3\}.
\end{align*}
Moreover, there exist $A,B\in\frak{so}(3)$, such that {{the set $\{\partial_A f,\partial_B f,\omega_1 f,\omega_2 f,\omega_3 f\}$ is linearly independent}} over $\mathbb C$.
\end{lemma}
The proof of Lemma \ref{lem:key} hinges on the fact that the codimension of a proper, nontrivial subalgebra of $\frak{so}(3)$ equals 2.
The linear independence of the set $\{\partial_A f,\partial_B f,\omega_1f,\omega_2f,\omega_3f\}$ follows from the facts that 
$\partial_A f,\partial_B f$ are real {\it even} functions, whereas 
$\omega_1f,\omega_2f,\omega_3f$ are real {\it odd} functions.

\subsubsection{Uniform random walks in $\mathbb R^3$}\label{sec:PT}
We will need explicit expressions for various convolution measures $\sigma^{\ast k}$.
These can be interpreted in terms of random walks, and as such are sometimes available in the probability theory literature.
More precisely, consider independent and identically distributed random variables $X_1,X_2,X_3$, taking values on $\mathbb S^{2}$ with uniform distribution. 
Then $Y_3=X_1+X_2+X_3$ is known as the {\it uniform 3-step random walk} in $\R^3$. 
If $p_3$ denotes the probability density of $|Y_3|$, then a straightforward computation in polar coordinates reveals that $(\sigma\ast\sigma\ast\sigma)(r)=\sigma(\mathbb S^{2})^2 {p_3(r)}{r^{-2}}$. 
Such considerations quickly lead to the following formulae for spherical convolutions.
\begin{lemma}[\cite{OSQ21a}]\label{lem:2-fold} The following identities hold:
	\begin{equation}\label{eq:doubledequal3}\notag
(\sigma\ast\sigma)(\xi)=\frac{2\pi}{|\xi|},\text{ if }|\xi|\leq2,
	\end{equation}
	\begin{equation}\label{eq:tripledequal3}\notag
	(\sigma\ast\sigma\ast\sigma)(\xi)=\begin{cases}
	8\pi^2,&\text{ if }|\xi|\leq 1,\\
	 4\pi^2\Bigl(\frac{3}{|\xi|}-1\Bigr),&\text{ if }1\leq |\xi|\leq 3.
	\end{cases}
	\end{equation}
\end{lemma}

\begin{corollary}\label{cor:sigma}
$\Phi_6[{\bf 1}]=2\pi\Phi_4[{\bf 1}].$
\end{corollary}
Indeed, from Lemma \ref{lem:2-fold} we have that
\begin{align*}
    \Phi_4[{\bf 1}]&=(2\pi)^3 \|{\bf 1}\|_{L^2(\mathbb S^{2})}^{-4}\|\sigma\ast\sigma\|_2^2=16\pi^4,\text{ and}\\ 
\Phi_6[{\bf 1}]&=(2\pi)^3  \|{\bf 1}\|_{L^2(\mathbb S^{2})}^{-6}\|\sigma\ast\sigma\ast\sigma\|_2^2=32\pi^5.
\end{align*}
\subsubsection{End of proof of Theorem \ref{thm:spheres} when $(d,p)=(2,6)$}
By Step 1, it suffices to check that any non-constant {\it critical point} (i.e.\@ an $L^2$ solution of the Euler--Lagrange equation \eqref{eq:EL}) $f:\mathbb S^{2}\to\mathbb C\in C^1(\mathbb S^{2})$ of $\Phi_6$ satisfies $\Phi_6[f]<\Phi_6[{\bf 1}]$.  
By Step 2, we may further assume that $f_\star=f$ is real-valued, and that $\|f\|_{L^2}=1$.
From $T_f(f)=\lambda f$, one checks that $\lambda=\Phi_6[f]$. 
Thus, by Steps 3--4,
\begin{equation}\label{eq:One}
\Phi_6[f]
=\lambda
=\tfrac12(\lambda+5\times\tfrac{\lambda}5)
<\frac12 \textup{tr}(T_f)
=2\pi \|\widehat{f\sigma}\|_4^4,
\end{equation} 
where the strict inequality is a consequence of Lemma \ref{lem:key} together with the fact that all eigenvalues of $T_f$ are strictly positive, and the last identity has been observed in \eqref{eq:trace}. 
But
\begin{equation}\label{eq:Two}
2\pi\|\widehat{f\sigma} \|_4^4= 2\pi\Phi_{4}[f]\leq  2\pi\Phi_{4}[{\bf 1}]=\Phi_{6}[{\bf 1}]
\end{equation} 
where the inequality follows from the result \eqref{e:foschi2} of Foschi \cite{Fo15} reviewed in Section \ref{s:circle}, and the last identity 
was seen in Step 5 (Corollary \ref{cor:sigma}).
From \eqref{eq:One} and  \eqref{eq:Two} 
it follows that $\Phi_6[f]<\Phi_6[{\bf 1}]$, and this concludes the sketch of the proof of the case $(d,p)=(2,6)$ of Theorem \ref{thm:spheres}.

\subsection{Higher dimensions and exponents}
The special case $(d,p)=(2,6)$ of Theorem \ref{thm:spheres} which we addressed in \S \ref{sec:specialcase}, while illustrative of the general scheme, relies on several crucial simplifications which made the proof sketch fit in just a few pages. 
In order to deal with general even exponents and different dimensions, further ideas and techniques are needed. 
These turn out to be broadly connected with the following areas:
\begin{itemize}
    \item[$\bullet$] {\it Non-commutative algebra.} When trying to generalize Lemma \ref{lem:key} to higher dimensions, one is naturally led to the following question: What is the minimal codimension of a proper subalgebra of $\frak{so}(d)$? The answer is known and reveals an interesting difference that occurs in the four-dimensional case: the minimal codimension of a proper subalgebra of $\frak{so}(d)$ equals $d-1$ if $d\geq 3, d\neq 4$, but equals $2$ if $d=4$. In group theoretical terms, the group $\textup{SO}(4)/\{\pm I\}$ is {\it not} simple, whereas all other groups $\textup{SO}(d)$ are simple (after modding out by $\{\pm I\}$ if $d$ is even). 
    \item[$\bullet$] {\it Combinatorial geometry.} When trying to extend the relevant estimates from Corollary \ref{cor:sigma} to the multilinear setting of $(p/2)$-fold spherical convolutions, one faces certain variants of the {\it cube slicing problem}: Given $0<k<d$, what is the maximal volume of the intersection of the unit cube $[-\frac 12,\frac12]^d$ with a $k$-dimensional subspace of $\R^d$? The cube slicing problem has been intensely studied, but a complete solution remains out of reach. Fortunately, the methods that have been developed for this problem can be adapted to fulfil our needs.
        \item[$\bullet$] {\it Analytic number theory.} The rather direct approach we took in \S\ref{sec:PT} needs to be refined in order to tackle higher dimensions. Uniform random walks in $\R^d$ are lurking in the background and, despite being a classic topic in probability theory, a complete answer in even dimensions remains a fascinating, largely unsolved problem, which via the theory of hypergeometric functions and modular forms exhibits some deep connections to analytic number theory \cite{BSWZ12}. In view of this, we combined known formulae for uniform random walks with rigorous numerical integration and asymptotic analysis for a certain family of weighted integrals in order to complete our task.
\end{itemize}

\subsection{Complex-valued maximizers}\label{sec:complex}
Once real-valued maximizers have been identified, one can proceed to characterize {\it all} complex-valued maximizers.
\begin{theorem}[\cite{OSQ21a}]\label{thm:complex}
Let $d\geq 1$ and $p\geq 2+\frac{4}{d}$ be an even integer.
Then each complex-valued maximizer of the functional $\Phi_{d,p}$ is of the form
\[c e^{i\xi\cdot\omega}F(\omega),\]
for some $\xi\in\R^d$, some $c\in\mathbb C\setminus\{0\}$, and some nonnegative maximizer $F$ of $\Phi_{d,p}$ satisfying $F(\omega)=F(-\omega)$, for every $\omega\in\mathbb S^{d}$.
\end{theorem}
\noindent
Our next result is an immediate consequence of Theorems \ref{thm:spheres} and \ref{thm:complex}.

\begin{corollary}[\cite{OSQ21a}]
	Let $d\in\{2,3,4,5,6\}$ and $p\geq 4$ be an even integer.
	Then  all complex-valued maximizers of the functional $\Phi_{d,p}$ are given by
	$$f(\omega)=ce^{i\xi\cdot\omega},$$
	for some $\xi\in\R^d$ and $c\in\mathbb C\setminus\{0\}$.
	The same conclusion holds for $d=1$ and even integers $p\geq 8$, provided that constants maximize $\Phi_{1,6}$.
\end{corollary}

\subsection{A sharp extension inequality on $\mathbb S^{7}$}
The study of sharp {\it weighted} spherical extension estimates is  linked to the question of stability of such estimates, and was very recently inaugurated in \cite{CNOS21}.
In particular, the sharp weighted extension inequality from \cite[Theorem 1]{CNOS21} leads to the following result which is the first instance of a sharp extension inequality on $\mathbb S^7$.
\begin{theorem}[\cite{CNOS21}]
    For every $a>\tfrac{2^{25}\pi^2}{5^2  7^2  11}$, the following sharp inequality holds:
    \begin{equation}\label{eq:corrector_estimate}
        \int_{\mathbb R^8} \lvert \widehat{f\sigma}(x)\rvert^4\, \textup d x +a\left\lvert\,\, \int_{\mathbb S^7}  f(\omega)\, \textup d \sigma(\omega)\right\rvert^4 \le {\bf W}_a \left(\,\, \int_{\mathbb S^7} \lvert f(\omega)\rvert^2\, \textup d \sigma(\omega)\right)^2
    \end{equation}
    with optimal constant given by
    \begin{equation*}
        {\bf W}_a=\int_{\mathbb R^8} \widehat{\sigma}(x)^4\frac{\textup d x}{\sigma(\mathbb S^7)^2} + a \sigma(\mathbb S^7)^2
    \end{equation*}
    Equality in \eqref{eq:corrector_estimate} occurs if and only if $f$ is constant on $\mathbb S^7$.
\end{theorem}
 We emphasize that constants are the unique {\it complex-valued} maximizers for \eqref{eq:corrector_estimate}, in contrast to the situation considered in \S \ref{sec:complex}.
 An interesting open problem is to lower the value of the threshold $\tfrac{2^{25}\pi^2}{5^2  7^2  11}$, hopefully all the way down to 0.

\section{Agmon--H\"ormander type estimates} 
\label{s:agmonh}
In this section, we describe some simple estimates for the  extension operator on spheres which, perhaps surprisingly, are {\it not always} maximized by constants.
For simplicity we restrict attention to the circle $\mathbb S^1\subset\R^2$. However,  analogous results have been recently proved in all dimensions $d\geq 1$; see \cite{NOS21}.

Our starting point is the Agmon--Hörmander estimate on the circle,
\begin{equation}\label{eq:AgHor}
\frac1R\int_{B_R} |\widehat{f\sigma}(x)|^2 \frac{\textup d x}{(2\pi)^2} \leq {\bf C}_R \int_{\mathbb S^1} |f(\omega) |^2\textup d \sigma(\omega),
\end{equation}
where $B_R\subset\R^2$ denotes a ball of arbitrary radius $R>0$ centered at the origin and $\sigma$ stands for the usual arclength measure on $\mathbb S^1$.
Agmon and Hörmander \cite{AH76} observed that \eqref{eq:AgHor} holds with a constant ${\bf C}_R$ that approaches $\frac1{\pi}$ as $R\to\infty$, but did not investigate its optimal value.
In Theorem \ref{thm:AH1} below, such optimal value is obtained in terms of the auxiliary quantities
\begin{equation}\label{eq:LamDef}
\Lambda_R^k:=\frac R2 J_k^2(R)-\frac R2 J_{k-1}(R)J_{k+1}(R),
\end{equation}
where $J_n$ denotes the usual Bessel function; see also Figure \ref{fig:AgHor}.
\begin{theorem}[\cite{NOS21}]\label{thm:AH1}
For each $R>0$,
\[{\bf C}_R
=\left\{ \begin{array}{ll}
\Lambda_R^0, & \text{if } (J_0J_1)(R)\geq 0,\\
\Lambda_R^1, & \text{if } (J_0J_1)(R)\leq 0.
\end{array} \right.
\]
The corresponding space of maximizers is given by
\[\mathcal M_R=\left\{ \begin{array}{ll}
\mathcal H_0, & \text{if }(J_0J_1)(R)> 0, \\
\mathcal H_1, & \text{if }(J_0J_1)(R)< 0,\\
\mathcal H_0\oplus \mathcal H_1, & \text{if }J_0(R)= 0,\\
\mathcal H_0\oplus \mathcal H_1\oplus \mathcal H_2, & \text{if }J_1(R)= 0.
\end{array} \right.\]
where $\mathcal H_k\subset L^2(\mathbb S^1)$ denotes the vector space of degree $k$ circular harmonics. 
\end{theorem}
One remarkable feature of Theorem \ref{thm:AH1} is that constants are seen to not always be maximizers, even though they are the unique functions which are invariant under the full rotational symmetry group of \eqref{eq:AgHor}.
In this way Theorem \ref{thm:AH1} identifies an instance of {\it symmetry breaking}, which to the best of our knowledge had not yet been observed for an estimate involving the spherical extension operator.

The next natural question concerns the {\it stability} of inequality \eqref{eq:AgHor},
which can be phrased in terms of lower bounds for the following {\it deficit functional}:
\[\delta_R[f]:={\bf C}_R \|f\|_{L^2(\mathbb S^1)}^2 - \frac1R\int_{B_R} |\widehat{f\sigma}(x)|^2 \frac{\textup d x}{(2\pi)^2}.\]
Clearly, $\delta_R[f]\geq 0$ for every $f\in L^2(\mathbb S^1)$ but, in the spirit of Bianchi--Egnell \cite{BE91}, more can be said; see also Figure \ref{fig:AgHor}, and recall the space $\mathcal M_R$ which has been defined in Theorem \ref{thm:AH1}.
\begin{theorem}[\cite{NOS21}]\label{thm:Stability}
The following sharp two-sided inequality holds:
\begin{equation}\label{eq:sharpenedAgHor}
{\bf S}_R\, \textup {dist}^2(f,\mathcal M_R)
\leq \delta_R[f]
\leq {\bf C}_R \, \textup{dist}^2(f,\mathcal M_R).
\end{equation}
Equality occurs in the right-hand side inequality of \eqref{eq:sharpenedAgHor} if and only if $f\in\mathcal M_R$. 
Equality occurs in the left-hand side inequality of \eqref{eq:sharpenedAgHor} if and only if $f\in\mathcal M_R\oplus \mathcal E_R$, where:

\begin{center}
\begin{tabular}{ c|c|c:c } 
 ${\bf S}_R=$ & $\mathcal E_R=$ & \textup{if} & \textup{and}  \\ 
 \hline
$\Lambda_R^0-\Lambda_R^1$ & $\mathcal H_1$  && $(J_1J_2)(R)>0$   \\ 
$=$
 & $\mathcal H_1\oplus\mathcal H_2\oplus \mathcal H_3$& $(J_0J_1)(R)>0$& $(J_1J_2)(R)=0$   \\ 
  $\Lambda_R^0-\Lambda_R^2$ & $\mathcal H_2$&&$(J_1J_2)(R)<0$  \\ 
\hline
 $\Lambda_R^1-\Lambda_R^0$ & $\mathcal H_0$&& $(J_0J_1+J_1J_2+J_2J_3)(R)>0$ \\
$=$ & $\mathcal H_0\oplus \mathcal H_3$ & $(J_0J_1)(R)<0$ & $(J_0J_1+J_1J_2+J_2J_3)(R)=0$ \\ 
$\Lambda_R^1-\Lambda_R^3$ & $\mathcal H_3$& &$(J_0J_1+J_1J_2+J_2J_3)(R)<0$ \\ 
\hline
$\Lambda_R^0-\Lambda_R^2$ & $\mathcal H_2$&$J_0(R)=0$ &$(J_2J_3)(R)>0$ \\ 
 $\Lambda_R^0-\Lambda_R^3$ & $\mathcal H_3$&&$(J_2J_3)(R)<0$ \\ 
\hline
$\Lambda_R^0-\Lambda_R^3$ & $\mathcal H_3$&$J_1(R)=0$&$(J_3J_4)(R)>0$   \\ 
 $\Lambda_R^0-\Lambda_R^4$ & $\mathcal H_4$&&$(J_3J_4)(R)<0$ \\
 \hline
\end{tabular}
\end{center}
\end{theorem}
\vspace{.5cm}

The proofs of Theorems \ref{thm:AH1} and \ref{thm:Stability} rely on two observations. Firstly, by orthogonality of the circular harmonic decomposition $f=\sum_{k\geq 0}  Y_k$, we have that
\[\frac1{R} \int_{B_R} \lvert \widehat{f\sigma}(x)\rvert^2\, \frac{\textup d x}{(2\pi)^2}=\sum_{k\geq 0} \Lambda_R^k \|Y_k\|_{L^2}^2,\]
where the $\Lambda_R^k$ have been defined in \eqref{eq:LamDef}.
Secondly, 
 ${\bf C}_R = \sup_{k\geq 0} \Lambda_R^k$, where $f$ attains the supremum if and only if $f=\sum_{} Y_{k_j}$, for some 
$k_j\in\{k\geq 0: \Lambda_R^k=\sup_h \Lambda_R^h\}$; see also \cite{BS17}.
In our case of interest, we can conveniently rewrite the $\Lambda_R^k$ in integral form,
\[\Lambda_R^k = \frac1R \int_0^R  J_k^2(r) r\textup d r,\]
and invoke certain well-known Bessel recursions to start gaining control on both extremal problems corresponding to Theorems  \ref{thm:AH1} and \ref{thm:Stability}.

The above explicit expression for the optimal constant ${\bf C}_R$ leads to the following {\it loss-of-regularity}
\footnote{Similar phenomena have been recently observed in the related setting of the Brascamp--Lieb inequalities \cite{BBCF17, BBFL18}.} 
statement: ${\bf C}_R$ is {\it not} a differentiable function of $R$ at each positive zero of $J_0J_1$, but it defines a Lipschitz function on $(0,\infty)$ which is real-analytic between any two consecutive zeros of $J_0J_1$. We note that such zeros precisely correspond to those values of $R$ at which a jump in the dimension of the space of maximizers $\mathcal M_R$ is observed. 

The behaviour of the stability constant ${\bf S}_R$ is also interesting. Since the deficit functional $\delta_R[f]$ clearly defines a continuous function of $R$, the quantity ${\bf S}_R$ must have a jump discontinuity at the positive zeros of $J_1J_2$, where $\textup {dist}(f, \mathcal M_R)$ likewise jumps. Moreover, the explicit expression for ${\bf S}_R$ can be used to establish that it defines a piecewise real-analytic function of $R$ between any two consecutive zeros of $J_0J_1$, which fails to be differentiable at each positive zero of $J_2$.

\begin{figure}
    \centering
    \includegraphics[width=.8\textwidth]{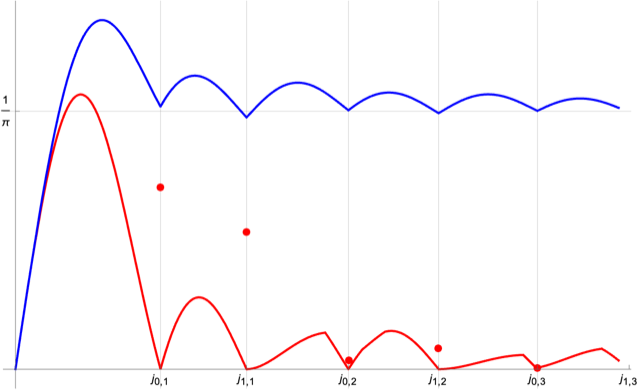}
    \caption{\footnotesize Optimal constant ${\bf C}_R$ (blue) and stability constant ${\bf S}_R$ (red) for the Agmon--Hörmander estimate on the circle when $0<R<10$.}
    \label{fig:AgHor}
\end{figure}

\section{Cones} 
\label{s:cones}
We consider for $d\ge 2$ the two-sheeted cone 
$$\mathbb{K}^{d}:=\{(\tau,\xi)\in \mathbb R \times \mathbb R^d: \tau^2=\lvert\xi\rvert^2\}.$$
It is the conical section
\eqref{e:plane} with
$\alpha=0$, $\beta=1$ and $\gamma=0$.
The measure $\sigma$ as in
\eqref{eq:sigma_measure_intro}
becomes 
$$\int _{\mathbb K^d} f \, \textup d\sigma =\int_{\mathbb R^{d+1}} f(\tau, \xi)\delta(\tau^2-|\xi|^2)
\textup d\tau \textup d\xi.$$
We split $\sigma=\sigma_++\sigma_-$ via
\begin{equation}\label{eqct:conic_measure_general}
    \begin{array}{cc}
    \delta(\tau^2-\lvert\xi\rvert^2)=\frac{\delta(\tau -\lvert\xi\rvert)}{2\lvert\xi\rvert}+
    \frac{\delta(\tau +\lvert\xi\rvert)}{2\lvert\xi\rvert}.
    \end{array}
\end{equation}
This follows from $\tau^2-\lvert\xi\rvert^2=(\tau-\lvert\xi\rvert)(\tau+\lvert\xi\rvert)$ by formal manipulations of delta calculus; see~\cite[Appendix~A]{FOS17}. The expressions of $\sigma_{\pm}$ involve the singular term $1/\lvert\xi\rvert$, but this is singular on a set of measure zero and it is locally integrable. We conclude that $\sigma_{\pm}$ are well-defined, and so is $\sigma$, with explicit formulae
\begin{equation}\label{eq:explicit_measure_cone}\notag
    \int_{\mathbb K^d} g(\tau, \xi)\, \textup d\sigma_\pm = \int_{\mathbb R^d} \frac{ g(\pm\lvert \xi\rvert, \xi)}{2\lvert \xi\rvert}\, \textup d\xi.
\end{equation}

By the Strichartz estimates~\cite{St77}, the  extension operator $\mathcal{E}f:=\widehat{f\sigma}$ defines a bounded linear map from $L^2(\sigma)$ into $L^p(\mathbb R^{1+d})$ for $p=2\frac{d+1}{d-1}$. This is the only value of $p$ for which such boundedness can hold, due to the scaling symmetry $\sigma_\pm(\lambda \tau, \lambda \xi)=\lvert\lambda\rvert^{-2}\sigma_\pm(\tau, \xi)$. 
Among all conic sections, only paraboloids have similar scaling symmetries as the cones. The
phenomenon of such scaling does not occur for spheres or hyperboloids discussed in this survey. 

We denote the conjectured maximizer of~\eqref{e:extfunctional} by
\begin{equation}\label{eq:Foschian}\notag
        \displaystyle f_\star(\tau, \xi):=e^{-\lvert \tau\rvert}.
\end{equation}
We will study whether $f_\star$ is a maximizer for $\lVert \mathcal E f\rVert_{L^p(\mathbb R^{1+d})}^p\lVert f\rVert_{L^2(\sigma)}^{-p}$ when $p=2\frac{d+1}{d-1}$. This is the $\mathbb K^d$-version of problem~\eqref{e:quotient} from the introduction. 

We will repeatedly use the property that $u(t, x)=\mathcal E f(t, x)$ is a solution of the wave equation. Indeed, by differentiating 
\begin{equation*}
     u(t, x)=\int_{\mathbb K^d} f(\tau, \xi)e^{-i(t\tau + x\cdot \xi)} \textup d(\sigma_++\sigma_-)(\tau, \xi)
\end{equation*}
we see that $u$ satisfies 
\begin{equation}\label{e:wave}
\partial_t^2 u =\Delta u,
\end{equation} 
with initial data 
\begin{equation}\label{eq:f_to_initial_data}
    \begin{array}{cc}
        \widehat{u}(0,\xi)=\frac{f(\lvert\xi\rvert, \xi)+f(-\lvert\xi\rvert, \xi)}{2\lvert\xi\rvert}, & \partial_t \widehat{u}(0,\xi)=\frac{-if(\lvert\xi\rvert, \xi)+if(-\lvert\xi\rvert, \xi)}{2},
    \end{array}
\end{equation}
where we have used the spatial Fourier transform  $\widehat{v}(t,\xi)=\int_{\mathbb R^d} v(t,x)e^{-ix\cdot \xi}\, \textup d x$. In particular, the conjectured maximizer extension $u_\star=\mathcal E f_\star$ satisfies $\widehat{u}_\star(0, \xi)=e^{-\lvert\xi\rvert}/\lvert\xi\rvert$ and $\partial_t\widehat{u}_\star(0, \xi)=0$. An explicit computation reveals that, for some $C_d>0$,
\begin{equation}\label{eq:Foschian_initial_data}
    \begin{array}{cc}
        u_\star(0, x)=C_d(1+\lvert x \rvert^2)^{\frac{1-d}{2}}, & \partial_t u_\star(0, x)=0.
    \end{array}
\end{equation}

\subsection{Criticality of the conjectured maximizer via the Penrose transform}
This subsection is based on \cite{Ne18}.
Define an injective map from $(t,r)\in {\mathbb R}^2$ to $(T,R)\in \R^2$ 
by
\begin{equation}\label{eq:Penrose_Coords}\notag
    \begin{array}{cc}
      T=\arctan(t+r)+\arctan(t-r), & R=\arctan(t+r)-\arctan(t-r).
    \end{array}
\end{equation}
This is essentially the composition of a rotation by $\pi/4$ of ${\mathbb R}^2$,
a component-wise arctangent, and the inverse rotation. 

We use this to define an injective map $\cal{P}$ from the Minkowski spacetime $\mathbb R^{1+d}$ into $[-\pi, \pi]\times \mathbb S^d$ as follows. We consider a generic $(t,x)\in\R^{1+d}$ in polar coordinates $(t,r,\omega)$, with $r=|x|$ and $\omega=x/r$, and then map this to $(T, \cos R, \omega \sin R )$. 
The last two components of the latter define a point
$X=(\cos R , \omega\sin R )$ on the sphere $\mathbb S^{d}$, as claimed;
see Figure~\ref{fig:penrose}.
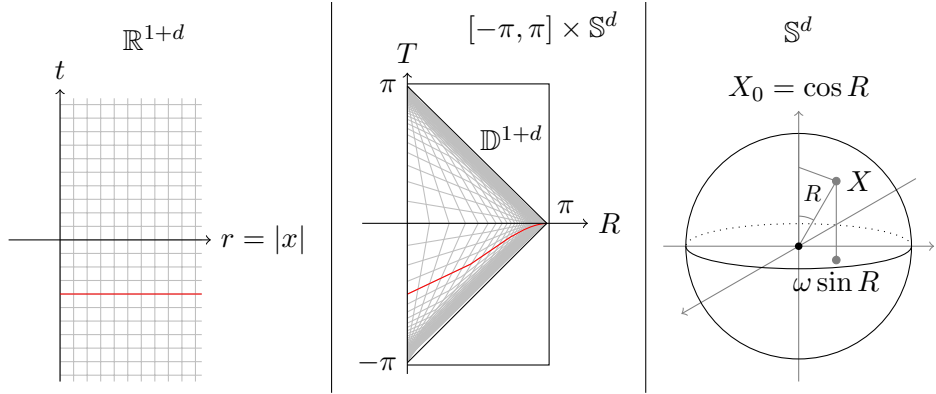
\begin{figure}
\centering
\begin{tabular}{c|c|c}
\begin{tikzpicture}[scale=0.6]
\draw[step=0.3, gray!50!white] (0, -3.14) grid (3.14, 3.14);
\draw[->] (-1.14, 0) -> (3.34,0) node[anchor=west] {$r=\lvert x \rvert$};
\draw[->] (0, -3.14) -> (0, 3.34) node[anchor=south]{$t$};
\draw[red] (0, -1.2) node[anchor=east] {}-- (3.14, -1.2) node[anchor=west]{} ;
\draw (3, 4.5) node[anchor=east] {$\mathbb R^{1+d}$};
\end{tikzpicture}
&
\begin{tikzpicture}[scale=0.6]
\foreach \t in {-8, -7.75, ..., 8}
    \draw[gray!50!white, domain=0:30] plot({rad(atan(\t+\x)-atan(\t-\x))}, {{rad(atan(\t+\x)+atan(\t-\x))}});
\foreach \r in {0, 0.25, ..., 5}
	\draw[gray!50!white, domain=-20:20] plot({rad(atan(\x+\r)- atan(\x-\r))}, {{rad(atan(\x+\r)+ atan(\x-\r))}});
	 \draw[red, domain=0:30] plot({rad(atan(-1+\x)- atan(-1-\x))}, {{rad(atan(-1+\x) + atan(-1-\x))}});
\draw[->] (-1, 0) -> (4,0) node[anchor=west] {$R$};
\draw[->] (0, -3.34) -> (0, 3.34) node[anchor=south]{$T$};
\draw (0, -pi+0.05) 
node[anchor=east] {$-\pi$} 
-- (pi-0.05, 0) 
node[anchor=south west]{$\pi$} 
-- (0, pi-0.1) 
node[anchor=east]{$\pi$}
;
\draw (pi/2-0.18, pi/2-0.18) node[anchor=south west]{$\mathbb D^{1+d}$};
\draw (3, 3.75) node[anchor=south] {$[-\pi, \pi]\times \mathbb{S}^d$};
\draw (0,-pi) -- (pi,-pi) -- (pi, pi-0.05) -- (0, pi-0.05);
\end{tikzpicture} 
&
\begin{tikzpicture}
            \draw[dotted] (1.5,0) arc (0:180:1.5 and 0.3);
            \draw[gray, ->] (-1.8, 0) -- (1.8,0) node[black, anchor=west]{$ $};
            \draw[gray, ->] (0, -1.8) -- (0, 1.8) node[black, anchor=south]{$X_0=\cos R$};
            \draw[gray, ->] (30:1.8) -- (30:-1.8) node[black, anchor=north east]{$ $};
            \draw (0, 0) circle (1.5);
            \draw (1.5,0) arc (0:-180:1.5 and 0.3);
            \draw[gray] (0,0)--(0, 1.5);
            \draw[gray, fill=gray] (0,0)--(60:1) circle (0.05) node[black, anchor=west]{$X$} -- (0, 1.05);
            \draw[gray, fill=gray] (60:1) -- +(0, -1.05) circle (0.05) node[black,  yshift=-8]{$\omega \sin R $};
            \draw[gray] (0,0.4) arc (90:60:0.4);
            \draw (95:0.5) node[anchor=south west, scale=0.8] {$R$};
            \draw (0, 2.585) node[anchor=south]{$\mathbb S^d$};
            \fill (0,0) circle (0.05);
    \end{tikzpicture}
\end{tabular}
\caption{\footnotesize The map $\mathcal P$ of $\mathbb R^{1+d}$ (left) onto the Penrose diamond $\mathbb D^{1+d}\subset [-\pi, \pi]\times \mathbb S^d$ (center). Here $R\in[0, \pi]$ denotes the radial coordinate on $\mathbb S^d$ (right).}
\label{fig:penrose}
\end{figure}

The map $\cal{P}$ was first introduced by Penrose~\cite{Pen64}, and its range is the following open submanifold of $[-\pi, \pi]\times \mathbb S^d$, known as the \emph{Penrose diamond}:
\begin{equation}\label{eq:penrose_diamond} 
    \mathbb D^{1+d}:=\{ (T, \cos R, \omega \sin R)\, :\, \omega\in \mathbb S^{d-1},\ R< \pi -\lvert T\rvert\}.
\end{equation}
The $(d+1)$-dimensional volume element of $\mathbb D^{1+d}$ reads $(\sin R)^{d-1}\textup d T \textup d R \textup d S(\omega)$, where $\textup d S(\omega)$ denotes the surface measure on $\mathbb S^{d-1}$. Similarly the volume element of $\mathbb R^{1+d}$ reads $r^{d-1}\textup d t \textup d r \textup d S(\omega)$. We compute the pushforward of this volume element via $\mathcal P$ as
\begin{equation}\label{eq:vol_element_pushforward}
    (\sin R)^{d-1}\textup d T \textup d R \textup d S(\omega)= 2^{d+1}[(1+(t+r)^2)(1+(t-r)^2)]^{-\frac{1+d}{2}}r^{d-1}\textup d t \textup d r \textup d S(\omega).
\end{equation}
An important feature of $\mathcal{P}$ is that it is conformal, meaning that
\begin{equation*}
    \textup d T^2 -\textup d R^2 -(\sin R)^2\textup d\omega^2 = 4[(1+(t+r)^2)(1+(t-r)^2)]^{-1}(\textup d t^2 -\textup d r^2 -r^2\textup d\omega^2).
\end{equation*}
Therefore we can construct solutions to the wave equation in $\mathbb R^{1+d}$ by pulling back solutions to the conformal wave equation 
\begin{equation*} 
    \partial_T^2 U = \Delta_{\mathbb S^d} U -\frac{(d-1)^2}{4} U
\end{equation*}
in $\mathbb D^{1+d}$; for details we refer to ~\cite[p.\@ 139]{Tz00}. As a consequence, for each spherical harmonic $Y_n=Y_n(X)$ on $\mathbb S^d$ of degree $n$, the  function 
\begin{equation}\label{eq:wave_solutions_family}
    u(t, x)=\left[(1+(t+r)^2)(1+(t-r)^2)\right]^{\frac{1-d}{4}}\!\!\!e^{-iT(n + \tfrac{d-1}{2})} Y_n(\cos R, \omega \sin R )
\end{equation}
satisfies the wave equation~\eqref{e:wave}. 

By the uniqueness of solutions to the wave equation, we see that $u=\mathcal Ef$ for $f\in L^2(\sigma)$, uniquely identified by solving the linear initial conditions~\eqref{eq:f_to_initial_data} for $f$.   
Taking a spherical harmonic of degree $n=0$, i.e.,  a constant function on $\mathbb S^d$, and considering the real part only, we obtain a new representation for the conjectured maximizer $u_\star=\mathcal Ef_\star$. Precisely,  
\begin{equation}\label{eq:conjectured_maximizer_penrose}\notag
    u_\star(t, x)=C_d \left[(1+(t+r)^2)(1+(t-r)^2)\right]^{\frac{1-d}{4}}\!\!\!\cos(T(\tfrac{d-1}{2})), 
\end{equation}
since this last expression is easily seen to satisfy \eqref{eq:Foschian_initial_data}.

Defining $C_\star$ via
\begin{equation}
    \label{e:optimal}\lVert \mathcal E f_\star \rVert_{L^{2\frac{d+1}{d-1}}({\mathbb R}^{1+d})}=C_{\star} \lVert f_\star\rVert_{L^2(\sigma)},
\end{equation}
we recall that the conjecture we are studying reads as
\begin{equation}\label{e:suboptimal}
 \lVert \mathcal Ef\rVert_{L^{2\frac{d+1}{d-1}}({\mathbb R}^{1+d})}\le C_{\star} \lVert f\rVert_{L^2(\sigma)}
\end{equation}
for all $f$.
We show that this conjecture is false 
in even dimension by perturbing 
$f_\star$ into $f_\star +\epsilon f$
on each side of \eqref{e:optimal}
and showing that for small $\epsilon$ the left-hand side of \eqref{e:suboptimal}
varies of linear order in $\epsilon$ while the right-hand side varies of quadratic order in $\epsilon$; compare with Theorem~\ref{thm:s_equals_one_conjecture} below. Therefore, by choosing the sign of small enough $\epsilon$ appropriately, one can create either strict inequality between both sides, therefore contradicting \eqref{e:suboptimal}.

We let $u=\mathcal Ef$ as in \eqref{eq:wave_solutions_family} with $n>0$. One can see that $f$ is orthogonal to $f_\star$ using orthogonality of the spherical harmonics 
$1$ and $Y_n$ and the isometries\footnote{Here we define $\|F\|_{H^s(\mathbb S^d)}^2=\int_{\mathbb S^d} \big|\big(\frac{(d-1)^2}4-\Delta_{\mathbb S^d}\big)^{s/2}F\big|^2\textup d S$.}
\begin{equation}\label{eq:hilbert_isometries}\notag
    \begin{split}
        \lVert f\rVert_{L^2(\sigma)}^2&=\lVert u(0,.)\rVert_{\dot{H}^{1/2}({\mathbb R}^d)}^2 + \lVert \partial_t u(0,.)\rVert_{\dot{H}^{-1/2}({\mathbb R}^d)}^2 \\
        &=\lVert Y_n\rVert_{H^{1/2}(\mathbb S^d)}^2 +\lVert Y_n\rVert_{H^{-1/2}(\mathbb S^d)}^2,
    \end{split}
\end{equation}
where the first identity follows from \eqref{eq:f_to_initial_data} and the second one is based on a computation with fractional integrals from conformal theory; see~\cite[eq.~(2)]{Morpurgo} and \eqref{eq:morpurgo} below.
The orthogonality between $f$ and $f_\star$ 
established the quadratic behaviour in $\epsilon$ of the right-hand side of \eqref{e:suboptimal}. 
To verify the linear behaviour in $\epsilon$ of the left-hand side, we choose $Y_n$ to be a real spherical harmonic and compute its first variation
(up to a nonzero scalar constant) as 
\begin{equation}\label{eq:derivative_lhs}
        \Re \int_{\mathbb R^{1+d}}\lvert \mathcal E f_\star\rvert^{p-2}\overline{\mathcal E f_\star} \mathcal E f .
        \end{equation}
With \eqref{eq:vol_element_pushforward} and \eqref{eq:wave_solutions_family}, we obtain
for \eqref{eq:derivative_lhs}
\begin{equation}\label{eq:derivative_lhs_continued}
    \begin{split}
        &2^{d+1}\int_{\mathbb D^{1+d}}\left| \cos(\tfrac{d-1}{2}T)\right|^{p-2}\cos(\tfrac{d-1}2 T) \cos(T(n+\tfrac{d-1}{2})) Y_n(X).
    \end{split}
\end{equation}
Obviously but crucially, $\frac{d-1}{2}$ is an integer if and only if $d$ is odd. If that is the case, then from the formula $Y_n(-X)=(-1)^n Y_n(X)$ we see that the integrand $U(T, X)$ of~\eqref{eq:derivative_lhs_continued} satisfies
\begin{equation}\label{eq:odd_symmetry}
    U(T+\pi, -X)=(-1)^\frac{d-1}{2}U(T, X), 
\end{equation}
Recalling the definition~\eqref{eq:penrose_diamond} of $\mathbb D^{1+d}$, it follows from this symmetry that
\begin{equation}\label{eq:integral_diamond}\notag
    \begin{split}
        &\int_{\mathbb D^{1+d}} U(T, \cos R, \omega \sin R ) =\frac{(-1)^\frac{d-1}2}{2}\int_{-\pi}^{\pi} \int_{\mathbb S^d} U(T,\cos R , \omega \sin R ) (\sin R)^{d-1}\textup d T \textup d R \textup d S(\omega)\\ 
        &=\frac{(-1)^{\tfrac{d-1}2}}{2}\int_{-\pi}^\pi \lvert \cos(\tfrac{d-1}{2}T)\rvert^{p-2}\cos(\tfrac{d-1}2 T) \cos(T(n+\tfrac{d-1}{2}))\, \textup d T \int_{\mathbb S^d} Y_n =0.
    \end{split}
\end{equation}
We used that the very last integral over $Y_n$ vanishes because $n>0$.
This suggests that $f_\star$ is a critical point. A refined version of this argument even proves that $f_\star$ is a local maximizer of $\lVert \mathcal E f\rVert_{L^p(\mathbb R^{1+d})}^p\lVert f\rVert_{L^2(\sigma)}^{-p}$; see~\cite[Theorem~1.1]{GN20}.

However, when $d$ is even, the symmetry~\eqref{eq:odd_symmetry} fails. In this case we can take an explicit spherical harmonic of degree $2$ and compute the variation~\eqref{eq:derivative_lhs}, which turns out to never vanish for every even $d$. We present a similar calculation in the next subsection. This proves that, in this case, $f_\star$ cannot be a maximizer.

Interestingly, Ramos~\cite{Ra12} proved that maximizers exist for $\mathcal E$ in arbitrary $d\ge 2$. So this is an instance where maximizers exist but they are not of the type suggested by the title of this survey.

\subsection{Related problems on one-sheeted cones}\label{s:1cone}
We consider now the one-sheeted cone 
\begin{equation}\label{eq:one-sheeted_cone}\notag
    \mathbb{K}^d_+:=\{ (\tau, \xi)\in \mathbb R\times \mathbb R^d\ :\ \tau=\lvert\xi\rvert\}.
\end{equation}
This is not a conic section in the strict sense discussed in the introduction, but it is a very
natural subset of the two-sheeted cone $\mathbb{K}^d$.

We further deviate from the strict setup  by considering  the more general family of measures supported on $\mathbb{K}^d_+$, 
\begin{equation}\label{eq:sigma_s}
    \begin{array}{cc}
        \sigma_s(\tau, \xi):=\frac{\delta(\tau-\lvert \xi \rvert)}{\lvert\xi\rvert^{2s}}, & s\in \left[ \frac12, \frac d2\right).
    \end{array}
\end{equation}
The measure $\sigma_+$ introduced in~\eqref{eqct:conic_measure_general} corresponds to the case $s=1/2$ of \eqref{eq:sigma_s}. 

By a standard interpolation of the aforementioned Strichartz estimates with Sobolev embeddings, as in~\cite{FVV12}, the Fourier extension operator $\mathcal{E}_s f:=\widehat{f\sigma_s}$ is bounded from $L^2(\sigma_s)$ into $L^{p_s}(\mathbb R^{1+d})$, where
\begin{equation}\label{eq:p_s}\notag
    p_s:=2\frac{d+1}{d-2s}.
\end{equation}
To compare with the notation of~\cite{BJOS17, FVV12}, note that
$$\mathcal E_s f(-t, -x)=(2\pi)^d e^{it\sqrt{-\Delta}}u_s(x)
=\int_{\mathbb R^d} e^{it\lvert\xi\rvert+ix\cdot \xi} \widehat{u}_s(\xi)\,\textup d\xi,$$
$$
\widehat{u}_s(\xi):=\frac{f(\lvert \xi\rvert, \xi)}{\lvert\xi\rvert^{2s}}. $$

In~\cite{BJOS17}, it is conjectured that 
\begin{equation}\label{eq:f_s_star}\notag
    f_{\star}^{(s)}(\tau, \xi):=\lvert\tau\rvert^{2s-1}e^{-\lvert \tau\rvert} 
\end{equation}
is a maximizer of 
\begin{equation}\label{eq:quotient_one_sheet} 
    \lVert \mathcal{E}_s f\rVert_{L^{p_s}(\mathbb R^{1+d})}^{p_s}\lVert f\rVert_{L^2(\sigma_s)}^{-p_s}
\end{equation}
if and only if
\begin{equation}\label{eq:conjectured_nec_suff_conds}\notag
    s\in\left\{\frac12, \frac{d-1}4\right\}.
\end{equation}
Note that $s=(d-1)/4$ is equivalent to $p_s=4$. 
This conjecture is open, but some partial results are available.

In the case $s=1/2$, Foschi~\cite{Fo07} proved that $f_\star^{(1/2)}$ is a maximizer for $d\in \{2, 3\}$. In the same case, for all $d\ge 2$, in~\cite{GN20} it is shown that $f_\star^{(1/2)}$ is a local maximizer. This is analogous to the situation for odd $d$ on the two-sheeted cone. However, in the present case of the one-sheeted cone there is no distinction between odd and even $d$.
We summarize the status of the $s=1/2$ case of this conjecture in Table~\ref{table:cones}, also for the case of the two-sheeted cone $\mathbb K^d$.

In the case $s=(d-1)/4$, in~\cite{BR13} it is proved that $f_\star^{(s)}$ is a maximizer for $d=5$. In the same case, for $d\ge 2$, in~\cite{BJO16} it is shown that $f_\star^{(s)}$ is a maximizer among all radially symmetric functions. We note that the $d=5$ case of the two-sheeted cone has also been treated  in~\cite{BR13}, and further extended to a stability inequality in~\cite{Ne22}.

Finally, in the case $s\notin \{1/2, (d-1)/4\}$, in~\cite{BJOS17} it is proved that $f_\star$ is not a maximizer provided that $p_s$ is an even integer. 

\begin{table}
    \centering
        \begin{tabular}{c|c|c}
            Spatial dim.~$d$ & $\mathbb{K}^{d}$ & $\mathbb{K}^{d}_+$ \\ 
            \toprule
            2 & NO & YES \\ 
            3 & YES & YES \\ 
            4, 6, 8, \ldots & NO & Local \\ 
            5, 7, 9, \ldots & Local & Local 
        \end{tabular}
        \caption{\footnotesize Is $e^{-|\tau|}=f_\star^{(1/2)}(\tau,\xi)$ a maximizer of the Fourier extension on $\mathbb K^d$ and $\mathbb K^d_+$ when $s=1/2$?}
        \label{table:cones}
\end{table}

We conclude with an original result, showing how the Penrose transform can be applied to settle the $s=1$ case of the previous conjecture. This extends  ~\cite[Theorem 2.1]{BJOS17}  because we do not require $p_1$ to be even.
\begin{theorem}\label{thm:s_equals_one_conjecture}
    Let $d\ge 3,  d\ne 5$. There exists $f\in L^2(\sigma_1)$ such that 
    \begin{equation}\label{eq:s_equals_one_conclusion_one}
            \lVert f_\star^{(1)} + \epsilon f\rVert_{L^2(\sigma_1)} - \lVert f_\star^{(1)}\rVert_{L^2(\sigma_1)} =O(\epsilon^2), 
    \end{equation} 
    while, for some $C\ne 0$,
    \begin{equation}\label{eq:s_equals_one_conclusion_two}
        \lVert \mathcal E_1 (f_\star^{(1)} + \epsilon f)\rVert_{L^{2\frac{d+1}{d-2}}(\mathbb R^{1+d})} - \lVert \mathcal E_1 f_\star^{(1)}\rVert_{L^{2\frac{d+1}{d-2}}(\mathbb R^{1+d})}=C\epsilon + o(\epsilon).
    \end{equation}
\end{theorem}
As before, it follows from Theorem \ref{thm:s_equals_one_conjecture} that $f_\star^{(1)}$ cannot be a maximizer of~\eqref{eq:quotient_one_sheet} for $s=1$. 
\begin{proof}[Proof of Theorem~\ref{thm:s_equals_one_conjecture}]
    As in the previous case of the two-sheeted cone,  $u(t, x)=\mathcal E_1 f(t, x)$ satisfies the wave equation~\eqref{e:wave}. The initial data $u(0,x)$, $\partial_t u(0,x)$ are related by the support assumption on $f$ on the one-sheeted cone and read as follows:
    \begin{equation}\label{eq:one_sheeted_initial_data}
            \widehat{u}(0, \xi)=\frac{ f(\lvert \xi\rvert, \xi)}{\lvert\xi\rvert^{2s}},  
    \end{equation}
    \begin{equation}\label{eq:one_sheeted_initial_data_two}
        \partial_t \widehat{u}(0, \xi)=-i\lvert\xi\rvert\widehat{u}(0, \xi),
    \end{equation}
    where we used the spatial Fourier transform $\widehat{v}(t,\xi)=\int_{\mathbb R^d} v(t,x)e^{-ix\cdot \xi}\, \textup d x$. 
    
    As in the previous subsection, we consider $u(t,x)$ given by \eqref{eq:wave_solutions_family}, which has already been seen to solve the wave equation. We claim that it also satisfies condition~\eqref{eq:one_sheeted_initial_data_two}. 
    Indeed, we see by direct computation that 
    \begin{equation}\label{eq:one_cone_family_initial_data}\notag
        \begin{array}{cc}
            u(0, x)=(1+r^2)^{\frac{1-d}{2}}Y_n(X), & \partial_t u(0, x)=-2i(n+\frac{d-1}{2})(1+r^2)^{-\frac{1+d}{2}} Y_n(X),
        \end{array}
    \end{equation}
    where $X=(\cos R, \omega \sin R)$. 
    Now notice that the claimed condition can be equivalently written as 
    \begin{equation}\label{eq:constraint_equivalent}
        \partial_t u(0, \cdot)=-i\sqrt{-\Delta}u(0, \cdot).
    \end{equation}
    By the aforementioned conformal formula~\cite[eq.~(2)]{Morpurgo},
    \begin{equation}\label{eq:morpurgo}
        \sqrt{-\Delta} u(0, x) = 2(1+r^2)^{-\frac{1+d}{2}}\left(\frac{(d-1)^2}{4}-\Delta_{\mathbb S^d}\right)^{\frac12} Y_n (X).
    \end{equation}
    Recalling  $-\Delta_{\mathbb S^d}Y_n=n(n+d-1)Y_n$, which implies 
    \begin{equation*}
        \left(\frac{(d-1)^2}{4}-\Delta_{\mathbb S^d}\right)^{\frac12} Y_n =\left( n + \frac{d-1}{2}\right) Y_n,
    \end{equation*}
    condition~\eqref{eq:constraint_equivalent} then follows immediately.
    
    Now we notice that the conjectured maximizer $u_\star:=\mathcal E_1 f_\star^{(1)}$ is such that
    \begin{equation*}
            \widehat{u}_\star(0, \xi)=\frac{ e^{-\lvert\xi\rvert}}{\lvert\xi\rvert}, 
    \end{equation*}
    so the same computation as in~\eqref{eq:Foschian_initial_data} yields 
    \begin{equation}\label{eq:one_sheeted_conj_maximizer_initial_data}\notag
            u_\star(0, x)=C_d(1+\lvert x \rvert^2)^\frac{1-d}{2}, 
    \end{equation}
    for some $C_d>0$. By uniqueness of solutions to the wave equation we conclude 
    \begin{equation}\label{eq:u_star_one_sheeted}\notag
        u_\star(t, x)=C_d[(1+(t+r)^2)(1+(t-r)^2)]^{\frac{1-d}{4}}e^{-iT\frac{d-1}{2}}.
    \end{equation}

    Now consider $u$ with $n=2$, where we choose $Y_2$ in~\eqref{eq:wave_solutions_family} to be
    \begin{equation}\label{eq:rodrigues}\notag
        Y_2(X_0, \ldots, X_d)=(1-X_0^2)^\frac{2-d}{2}\frac{\textup d^2}{\textup d X_0^2}\left[ (1-X_0)^{2+\frac{d-2}{2}}\right].
    \end{equation}
    By  Rodrigues formula, this is indeed a spherical harmonic of degree $2$. We let $f$ be the corresponding function on the one-sheeted cone defined by the initial condition~\eqref{eq:one_sheeted_initial_data}. The same argument as in the previous subsection shows that this $f$ just constructed is orthogonal to $f_\star^{(1)}$, from which the quadratic variation~\eqref{eq:s_equals_one_conclusion_one} immediately follows.
    
    On the other hand, the left-hand side of~\eqref{eq:s_equals_one_conclusion_two} reads to first order in $\epsilon$  
    \begin{equation}\notag
        \epsilon\Re \int_{\mathbb R^{1+d}} \lvert u_\star\rvert^{2\frac{d+1}{d-2}-2}\overline{u_\star} u\, \textup d t \textup d x
    \end{equation}
    up to a nonzero scalar constant. Changing variables via $\mathcal P$, this is seen to equal
    \begin{equation}\label{eq:first_variation_one_cone}
    \epsilon\Re\int_{\mathbb D^{1+d}}\lvert \cos T+X_0\rvert^{\frac{d+1}{d-2}}e^{i2T}Y_2(X)(1-X_0^2)^{\frac{d-2}{2}} \textup d T \textup d X_0\textup d S(\omega).
    \end{equation}
    Since $Y_2$ depends on $X$ only via $X_0$, the integral in $\textup d S(\omega)$ is constant. The integrand $U(T, X_0)$  clearly satisfies $U(T+\pi, -X_0)=U(T, X_0)$. So, arguing as in the previous subsection, we conclude that~\eqref{eq:first_variation_one_cone} equals
    \begin{equation}\label{eq:right_hand_integral}
        \frac C2\int_{-\pi}^\pi\int_{-1}^1 \lvert \cos T + X_0\rvert^\frac{d+1}{d-2}\cos(2T)\frac{\textup d^2}{\textup d X_0^2}\left[(1-X_0^2)^{2+\frac{d-2}{2}}\right]\, \textup d X_0 \textup d T.
    \end{equation}

    It remains to prove that this integral does not vanish. To do so, let 
    \begin{equation*}
        h_d(\cos T):=\int_{-1}^1\lvert \cos T+X_0\rvert^{\frac{d+1}{d-2}}\frac{\textup d^2}{\textup d X_0^2}\left[(1-X_0^2)^{2+\frac{d-2}{2}}\right]\, \textup d X_0.
    \end{equation*}
    By partial integration, we see that 
    \begin{equation}\label{eq:h_d_function}
        h_d(\cos T)=\int_{-1}^1\lvert \cos T+X_0\rvert^{\frac{5-d}{d-2}} (1-X_0^2)^{2+\frac{d-2}{2}}\, \textup d X_0,
    \end{equation}
    up to a positive scalar constant. 
    
    We pause to compare this with the computation needed in the last paragraph of the previous subsection. There we did not have the weight $\lvert \cos T + X_0\rvert^\frac{d+1}{d-2}$, but we did not have the symmetry $U(T+\pi, -X_0)=U(T, X_0)$ either. 
    Therefore the partial integration would not result in a term like \eqref{eq:h_d_function} but instead yield some boundary terms
    due to the more complicated region of integration.

    Note that $h_d$ is a constant function for $d=5$, which is the only case left out of the statement of Theorem~\ref{thm:s_equals_one_conjecture}. In this case,~\eqref{eq:right_hand_integral} is readily seen to vanish, as it reduces to a constant multiple of the integral of $\cos(2T)$ over a full period.
    
    We claim that $h_d=h_d(y)$ defines a {\it strictly increasing} function of $y\in(0,1)$ if $d\in\{3,4\}$, and a {\it strictly decreasing} function of $y\in(0,1)$ if $d\geq 6$. 
    Define variables $x:=X_0$, $t:=x+y$. This yields 
    \[h_d(y)=\int_{y-1}^{y+1} {(1-(t-y)^2)^{\frac{d+2}2}}{|t|^\frac{5-d}{d-2}} \,\textup d t .\]
    Differentiating under the integral sign and using the fact that the integrand vanishes at the boundary,
    \[\frac {h_d'(y)}{d+2}= \int_{y-1}^{y+1} {(t-y)(1-(t-y)^2)^{\frac d2}}{|t|^\frac{5-d}{d-2}}\,\textup d t=\int_{-1}^1{x(1-x^2)^{\frac  d2}}{|y+x|^{\frac{5-d}{d-2}}}\,\textup d x.\]
    Breaking up the region of integration into two, and in one of them changing $x$ to $-x$  reveals that the last display equals
    \[ \int_0^1 x(1-x^2)^{\frac d2}\left({|y+x|^{\frac{5-d}{d-2}}}-{|y-x|^{\frac{5-d}{d-2}}}\right)\, \textup d x.\]
    Since $|y+x|>|y-x|> 0$ for every $(y,x)\in(0,1)^2$, the bracketed term is negative for $d> 5$ and positive for $d\in\{3, 4\}$. This proves the claim.

    We can finally conclude that~\eqref{eq:right_hand_integral} is nonzero. Indeed, the integrand in~\eqref{eq:right_hand_integral} satisfies $U(-T,X)=U(T, X)=U(\pi-T, -X)$, hence integration in $T$ can be taken over the interval $[0, \pi/2]$.
    So~\eqref{eq:right_hand_integral} equals, with partial integration,
    \begin{equation*}
        \int_0^{\pi/2}\cos(2T)h_d(\cos T)\, dT= 
                \frac12\int_0^{\pi/2}\sin(2T)\sin(T)h_d'(\cos T)\, \textup d T,
    \end{equation*}
    and the right-hand integral has a definite sign. 
    This concludes the proof of Theorem~\ref{thm:s_equals_one_conjecture}.
\end{proof}

\section{Hyperboloids} 
\label{s:hyperbolic}
Hyperboloids locally look like spheres, with largest curvature at the origin, and globally resemble cones.
As such, sharp restriction theory on hyperboloids shares features from both spheres and cones, and serves as a natural bridge between  Sections \ref{s:circle}--\ref{s:agmonh} and  \ref{s:cones}.
On the other hand,  genuinely new phenomena emerge, as we shall see.
Consider the upper sheet of the two-sheeted hyperboloid
\[\mathbb H^d=\{(\tau,\xi)\in\R\times\R^d: \tau=\langle\xi\rangle\},\]
where $\langle\xi\rangle:=\sqrt{1+|\xi|^2}$, 
corresponding to the case $(\alpha,\beta,\gamma)=(0,1,1)$ of \eqref{e:plane},
and equip it with the Lorentz invariant measure
\begin{equation}\label{eq:LorentzInv}
\textup d \sigma(\tau,\xi)= \delta(\tau-\langle\xi\rangle)\frac{\textup d \tau\textup d \xi}{\langle\xi\rangle}
\end{equation}
corresponding to \eqref{eq:sigma_measure_intro}.
The extension operator on $\mathbb H^d$ is given by
\[\mathcal Ef(t,x)=\int_{\R^d} e^{i(t,x)\cdot(\langle\xi\rangle,\xi)} f(\xi)\frac{\textup d \xi}{\langle\xi\rangle},\]
and since the 1977 work of Strichartz \cite{St77} it is known that
\begin{equation}\label{eq:StriHyp}
\|\mathcal Ef\|_{L^p(\R^{1+d})}\leq{\bf H}_{d,p}\|f\|_{L^2(\mathbb H^d)}
\end{equation}
as long as\footnote{With the caveat that the endpoint case $p=\infty$ has to be excluded when $d=1$.}
\begin{equation}\label{eq:ExpRangeHyp}
2+\frac4d\leq p\leq2+\frac4{d-1} \text{ if } d\geq 1.
\end{equation}
Note that the endpoints in the latter range correspond to the endpoint exponents for the sphere and cone, respectively; recall Figure \ref{fig:PDP}.
 ${\bf H}_{d,p}$ denotes the optimal constant in inequality \eqref{eq:StriHyp} and, according to \eqref{eq:LorentzInv}, 
$\|f\|^2_{L^2(\mathbb H^d)}=\,\,\int_{\mathbb R^d}|f(\xi)|^2\frac{\textup d \xi}{\langle\xi\rangle}.$
The extension operator on $\mathbb H^d$ naturally relates to the Klein--Gordon equation, $\partial_t^2u=\Delta u-u$.
This connection comes via the Klein--Gordon propagator, 
\[e^{it\sqrt{1-\Delta}}g(x)=\frac1{(2\pi)^d} \int_{\R^d} e^{i(t,x)\cdot(\langle\xi\rangle,\xi)} \widehat{g}(\xi)\textup d \xi,\]
together with the observation that $\mathcal E f(t,x)= (2\pi)^d e^{it\sqrt{1-\Delta}}g(x)$ as long as $\widehat{g}(\xi)=\langle\xi\rangle^{-1}\widehat{f}(\xi)$. 
This relation implies that estimate \eqref{eq:StriHyp} can be equivalently rewritten as 
\[\|e^{it\sqrt{1-\Delta}}g\|_{L^p_{t,x}(\R\times\R^d)}\leq (2\pi)^{-d}{\bf H}_{d,p}\|g\|_{H^{1/2}(\mathbb R^d)}
\]
where $\|\cdot\|_{H^{1/2}(\mathbb R^d)}$ denotes the usual nonhomogeneous Sobolev norm.

Quilodrán \cite{Q15} investigated some sharp instances of inequality \eqref{eq:StriHyp}, 
all corresponding to algebraic endpoints of the range \eqref{eq:ExpRangeHyp}, and proved that
\[{\bf H}_{2,4}=2^{3/4}\pi,\,\, {\bf H}_{2,6}=(2\pi)^{5/2},\,\, {\bf H}_{3,4}=(2\pi)^{5/4},\]
even though maximizers do not exist.
He also asked about the value of ${\bf H}_{1,6}$, which corresponds to the last algebraic endpoint question, and whether maximizers exist in the non-endpoint case in all dimensions.\footnote{The methods of \cite{FVV12} do not apply directly due to the lack of exact scaling invariance.}
Both of these questions were recently answered in \cite{COSS19a, COSSS18}.
\begin{theorem}[\cite{COSS19a}]\label{thm:Hyp1}
    ${\bf H}_{1,6}=3^{-1/12}(2\pi)^{1/2}$, and maximizers for \eqref{eq:StriHyp} do not exist if $(d,p)=(1,6)$.
\end{theorem}
\begin{theorem}[\cite{COSS19a, COSSS18}]\label{thm:Hyp2}
Maximizers for \eqref{eq:StriHyp} exist when $6< p<\infty$ if $d=1$ and 
$2+\frac4d< p<2+\frac4{d-1} \text{ if } d\geq 2.$
In fact, given any maximizing sequence $\{f_n\}$, there exist symmetries $S_n$ such that $\{S_nf_n\}$ converges in $L^2(\mathbb H^d)$ to a maximizer $f$, after passing to a subsequence.
\end{theorem}
The proof of Theorem \ref{thm:Hyp1} relies on two ingredients.
Firstly, by Lorentz invariance it suffices to study the convolution measure $(\sigma\ast\sigma\ast\sigma)(\tau,\xi)$ along the axis $\xi=0$, and we 
verify that $\tau\mapsto(\sigma\ast\sigma\ast\sigma)(\tau,0)$ defines a continuous function on the half-line $\tau>3$, which extends continuously to the boundary of its support, so that
\begin{equation}\label{eq:StrictMax}
\sup_{\tau>3} (\sigma\ast\sigma\ast\sigma)(\tau,0)=(\sigma\ast\sigma\ast\sigma)(3,0)=\frac{2\pi}{\sqrt{3}},
\end{equation}
and that this global maximum is strict.
By an application of the Cauchy--Schwarz inequality, which is similar but simpler than the one for the sphere discussed in Section  \ref{s:circle}, it follows that ${\bf H}_{1,6}\leq 3^{-1/12}(2\pi)^{1/2}$.
For the reverse inequality,
one checks that $f_n=\exp(-n\langle\cdot\rangle)$ forms a maximizing sequence for \eqref{eq:StriHyp}, in the sense that 
\[\lim_{n\to\infty} \frac{\|f_n\sigma\ast f_n\sigma\ast f_n\sigma\|_{L^2(\R^2)}^2}{\|f_n\|_{L^2(\mathbb H^1)}^6}=\frac{2\pi}{\sqrt{3}}.\]
This crucially relies on the fact that the strict global maximum of $\sigma\ast\sigma\ast\sigma$ occurs at the {\it boundary} of the support of the convolution; recall \eqref{eq:StrictMax}. 
In particular, maximizers for \eqref{eq:StriHyp} do not exist if $(d,p)=(1,6)$.

For higher-order convolutions $\sigma^{\ast k}$, $k\geq 4$, the global maximum occurs in the interior of the support, which already hints towards Theorem \ref{thm:Hyp2}.
The proof of this theorem is more involved, and crucially relies on a {\it refined Strichartz estimate}.
If $d\in\{1,2\}$, corresponding to the range \eqref{eq:ExpRangeHyp} whose endpoints are even integers, then this can be obtained via elementary methods, such as the Hausdorff--Young and Hardy--Littlewood--Sobolev inequalities.
For instance, on $\mathbb H^1$ it is proved in \cite[Cor.\@ 13]{COSS19a} that, for each $6\leq p<\infty$, there exists $C_p<\infty$ such that
\begin{equation}\label{eq:sharpenedSt}
\|\mathcal E f\|_{L^p(\R^2)} \leq C_p \sup_{k\in\mathbb Z} \|f_k\|_{L^2(\mathbb H^1)}^{1/3} \|f\|_{L^2(\mathbb H^1)}^{2/3}.
\end{equation}
The decomposition $f=\sum_{k\in\mathbb Z} f_k$ is such that $f_k=f 1_{\mathcal C_k}$, where the family of {\it hyperbolic caps} $\{\mathcal C_k\}_{k\in\mathbb Z}\subset\mathbb H^1$ is given by
\[\mathcal C_k :=\{(\tau,\xi)\in\mathbb H^1: \sinh(k-\tfrac12)\leq \xi\leq \sinh(k+\tfrac12)\}.\]
Inequality \eqref{eq:sharpenedSt} allows us to start gaining  control over arbitrary maximizing sequences. 
In particular, it can be used to show the existence of a {\it distinguished cap} which contains a positive proportion of the total mass; possibly after a Lorentz boost, the distinguished cap can be assumed to coincide with $\mathcal C_0$. This rules out the possibility of mass concentration at infinity, which had been previously identified in \cite{Q15} as the main obstruction to the precompactness of maximizing sequences modulo symmetries.
If $d\geq 3$, then the refined Strichartz estimate follows from bilinear restriction theory; see \cite[Theorem 5.1]{COSSS18}. 
We omit the technical details, and refer the interested reader to \cite[\S 5]{COSSS18}.

It would be interesting to understand whether the two-sheeted hyperboloid shares similar features to the ones described for the two-sheeted cone in Section \ref{s:cones}.

\section*{Acknowledgements}
GN and DOS were supported by the EPSRC New Investigator Award “Sharp Fourier Restriction Theory”, grant no. EP/T001364/1. 
DOS and CT acknowledge partial support from the Deutsche Forschungsgemeinschaft under Germany’s Excellence Strategy – EXC-2047/1 – 390685813, and (CT) from CRC 1060.
This work was initiated during a pleasant visit of CT to Instituto Superior Técnico, Universidade de Lisboa, whose hospitality is greatly appreciated.

\end{document}